\documentclass[11pt]{article}
\usepackage{amsmath, amsthm, amssymb, amsfonts, enumerate, xspace}

\newtheorem{theorem}{Theorem}[section]
\newtheorem{proposition}[theorem]{Proposition}
\newtheorem{corollary}[theorem]{Corollary}
\newtheorem{lemma}[theorem]{Lemma}

\newtheorem{fact}[theorem]{Fact}

\newtheorem{definition}[theorem]{Definition}
\newtheorem*{definition-notag}{Definition}

\theoremstyle{remark}
\newtheorem{remark}[theorem]{Remark}

\numberwithin{equation}{section}

\newcommand{\Const}[1]{C_{\text{\tiny\ref{#1}}}}
\newcommand{\const}[1]{c_{\text{\tiny\ref{#1}}}}

\DeclareMathOperator{\dist}{dist}
\DeclareMathOperator{\Span}{span}
\DeclareMathOperator{\supp}{supp}
\DeclareMathOperator{\Comp}{Comp}
\DeclareMathOperator{\Incomp}{Incomp}
\DeclareMathOperator{\spread}{spread}

\def \R {\mathbb{R}}

\def \Z {\mathbb{Z}}
\def \E {\mathbb{E}}
\def \P {\mathbb{P}}
\def \one {{\bf 1}}
\def \EE {\mathcal{E}}
\def \LL {\mathcal{L}}
\def \MM {\mathcal{M}}
\def \NN {\mathcal{N}}

\def \b {\beta}

\def \e {\varepsilon}
\def \d {\delta}
\def \l {\lambda}
\def \< {\langle}
\def \> {\rangle}
\def \HS {{\rm{HS}}}

\def \Dhat {\widehat{D}}
\def \H {{\bf (H)}\xspace}
\def \A {{\bf (A)}\xspace}

\begin{document}

\title{Invertibility of symmetric random matrices}
  
\author{Roman Vershynin\footnote{Partially supported by NSF grant DMS 1001829} \\
  University of Michigan\\
  \texttt{romanv@umich.edu}
  }

\date{February 1, 2011; last revised March 16, 2012}

\maketitle

\begin{abstract}
We study $n \times n$ symmetric random matrices $H$, possibly discrete, with iid above-diagonal entries. We show that $H$ is singular with probability at most $\exp(-n^c)$, and $\|H^{-1}\| = O(\sqrt{n})$.  Furthermore, the spectrum of $H$ is delocalized on the optimal scale $o(n^{-1/2})$. These results improve upon a polynomial singularity bound due to Costello, Tao and Vu, and they generalize, up to constant factors, results of Tao and Vu, and Erd\"os, Schlein and Yau.

\medskip

\end{abstract}

\tableofcontents

\section{Introduction}

\subsection{Invertibility problem}

This work is motivated by the invertibility problem for $n \times n$ random matrices $H$.
This problem consists of two questions: 

\begin{enumerate}
  \item What is the singularity probability $\P \{ H \text{ is singular} \}$?
  \item What is the typical value of the spectral norm of the inverse, $\|H^{-1}\|$?
\end{enumerate}

A motivating example is for random Bernoulli matrices $B$ whose entries are $\pm 1$ valued
symmetric random variables. If all entries are independent, it is conjectured that the singularity probability 
of $B$ is $(\frac{1}{2} + o(1))^n$, while the best current bound $(\frac{1}{\sqrt{2}} + o(1))^n$
is due to Bourgain, Vu and Wood \cite{BVW}. The typical norm of the inverse in this case is 
$\|B^{-1}\| = O(\sqrt{n})$ \cite{RV square, TV smallest}, see \cite{RV ICM}. Moreover, 
the following inequality due to Rudelson and the author \cite{RV square} simultaneously establishes the exponentially 
small singularity probability and the correct order for the norm of the inverse: 
\begin{equation}							\label{iid lowest}
\P \Big\{ \min_k s_k(B) \le \e n^{-1/2} \Big\} \le C \e + 2 e^{-cn}, 
\end{equation}
where $C,c>0$ are absolute constants. Here $s_k(B)$ denote the singular values of $B$, 
so the matrix $B$ is singular iff $\min_k s_k(B) = 0$; otherwise $\min_k s_k(B) = 1/\|B^{-1}\|$.
 
Less is known about the invertibility problem for {\em symmetric} Bernoulli matrices $H$, 
where the entries on and above the diagonal are independent $\pm 1$ valued
symmetric random variables. As is the previous case of iid entries, it is even difficult 
to show that the singularity probability converges to zero as $n \to \infty$. This was 
done by Costello, Tao and Vu \cite{CTV} who showed that
\begin{equation}							\label{CTV}
\P \big\{ H \text{ is singular} \big\} = O(n^{-1/8 + \d})
\end{equation}
for every $\d>0$. They conjectured that the optimal singularity probability bound is
for symmetric Bernoulli matrices is again $(\frac{1}{2} + o(1))^n$.

\subsection{Main result}

In this paper, we establish a version of \eqref{iid lowest} for symmetric random matrices. 
To give a simple specific example, our result will yield both an exponential bound on the singularity probability
and the correct order of the norm of the inverse for symmetric Bernoulli matrices:
$$
\P \big\{ H \text{ is singular} \big\} \le 2e^{-n^c}; \qquad
\P \big\{ \|H^{-1}\| \le C \sqrt{n} \big\} \ge .99
$$
where $C,c>0$ are absolute constants. 

Our results will apply not just for Bernoulli matrices, but also for general
matrices $H$ that satisfy the following set of assumptions:

\begin{itemize}
  \item[\H] $H = (h_{ij})$ is a real symmetric matrix. The above-diagonal entries $h_{ij}$, $i<j$,
    are independent and identically distributed random variables with zero mean and unit variance. 
    The diagonal entries $h_{ii}$ 
    can be arbitrary numbers (either non-random, or random but independent of the off-diagonal entries). 
\end{itemize}

The eigenvalues of $H$ in a non-decreasing order are denoted by $\l_k(H)$.

\begin{theorem}[Main]		\label{delocalization}
  Let $H$ be an $n \times n$ symmetric random matrix satisfying \H
  and whose off-diagonal entries have finite fourth moment.
  Let $K > 0$. Then for every $z \in \R$ and $\e \ge 0$, one has
  \begin{equation}										\label{eq delocalization}
  \P \Big\{ \min_k |\l_k(H) - z| \le \e n^{-1/2} 
    \text{ and } \max_k |\l_k(H)| \le K\sqrt{n} \Big\} \le C \e^{1/9} + 2e^{-n^c}.
  \end{equation}
  Here $C,c>0$ depend only on the fourth moment of the entries of $H$ and on $K$.
\end{theorem}

The bound on the spectral norm $\|H\| = \max_k |\l_k(H)|$ can often be removed
from \eqref{eq delocalization} at no cost, 
as one always has $\|H\| = O(\sqrt{n})$ with high probability
under the four moment assumptions of Theorem~\ref{delocalization}; see Theorem~\ref{delocalization 4} for a
general result.

Moreover, for some ensembles of random matrices one has $\|H\| = O(\sqrt{n})$ with exponentially high probability. 
This holds under the higher moment assumption that
\begin{equation}							\label{subgaussian}
\E \exp(h_{ij}^2/M^2) \le e, \quad i \ne j
\end{equation}
for some number $M>0$. Such random variables $h_{ij}$ are called {\em sub-gaussian} random variables,
and the minimal number $M$ is called the sub-gaussian moment of $h_{ij}$. The class of sub-gaussian
random variables contains standard normal, Bernoulli, and generally all bounded random variables, 
see \cite{V intro rmt} for more information. 
For matrices with subgaussian entries, it is known that $\|H\| = O(\sqrt{n})$
with probability at least $1-2e^{-n}$, see Lemma~\ref{norm subgaussian}. 
Thus Theorem~\ref{delocalization} implies:

\begin{theorem}[Subgaussian]				\label{delocalization subgaussian}
  Let $H$ be an $n \times n$ symmetric random matrix satisfying \H,
  whose off-diagonal entries are subgaussian random variables, and 
  whose diagonal entries satisfy $|h_{ii}| \le K \sqrt{n}$ for some $K$.
  Then for every $z \in \R$ and $\e \ge 0$, one has
  \begin{equation}							\label{eq delocalization subgaussian}
  \P \Big\{ \min_k |\l_k(H) - z| \le \e n^{-1/2} \Big\} \le C \e^{1/9} + 2e^{-n^c}.
  \end{equation}  
  Here $c>0$ and $C$ depend only on the sub-gaussian moment $M$ and on $K$.
\end{theorem}

\paragraph{Singularity and invertibility.} 
For $\e = 0$, Theorem~\ref{delocalization subgaussian} yields an exponential bound 
on singularity probability:
$$
\P \big\{ H \text{ is singular} \big\} \le 2e^{-n^{c}}.
$$
Furthermore, since $\min_k |\l_k(H) - z| = \|(H-zI)^{-1}\|$, \eqref{eq delocalization subgaussian}
can be stated as a bound on the spectral norm of the resolvent,
$$
\P \Big\{  \|(H-zI)^{-1}\| \ge \frac{\sqrt{n}}{\e} \Big\} \le C \e^{1/9} + 2e^{-n^c}.
$$
This estimate is valid for all $z \in \R$ and all $\e \ge 0$. In particular, we have
\begin{equation}							\label{resolvent whp}
\quad \|(H-zI)^{-1}\| = O(\sqrt{n}) \text{ with high probability}.
\end{equation}
For $z=0$ this yields the bound on the norm of the inverse, and on the {\em condition number} of $H$:
\begin{equation}										\label{condition number}
\|H^{-1}\| = O(\sqrt{n}), \quad \kappa(H) := \|H\| \|H^{-1}\| = O(n) \text{ with high probability}.
\end{equation}
In these estimates, the constants implicit in $O(\cdot)$ depend 
only on $M$, $K$ and the desired probability level.

\paragraph{Delocalization of eigenvalues.} 
Theorem~\ref{delocalization subgaussian} is a statement about delocalization of eigenvalues of $H$. 
It states that, for any fixed short interval $I \subseteq \R$ of length $|I| = o(n^{-1/2})$,
there are no eigenvalues in $I$ with high probability. This is consistent with the simple heuristics
about eigenvalue spacings. 
According to the spectral norm bound, all $n$ eigenvalues of $H$ lie in the interval of length $O(\sqrt{n})$.
So the average spacing between the eigenvalues is of the order $n^{-1/2}$. 
Theorem~\ref{delocalization subgaussian} states that, indeed, any interval of smaller length $o(n^{-1/2})$
is likely to fall in a gap between consequtive eigenvalues. For results in the converse direction,
on good localization of eigenvalues around their means, see \cite{TV localization} and the references therein.

\paragraph{Related results.}
A result of the type of Theorem~\ref{delocalization subgaussian} was known for 
random matrices $H$ whose entries have continuous distributions with certain smoothness properties, 
and in the bulk of spectrum, i.e. for $|z| \le (2-\d) \sqrt{n}$ (and assuming that the diagonal entries of $H$
are independent random variables with zero mean and unit variance). 
A result of Erd\"os, Schlein and Yau \cite{ESY} (stated for complex Hermitian matrices) is that 
\begin{equation}							\label{ESY}
\P \Big\{ \min_k |\l_k(H) - z| \le \e n^{-1/2} \Big\} \le C \e.
\end{equation}
This estimate does not have a singularity probability term $2e^{-n^c}$ 
that appears in \eqref{eq delocalization subgaussian}, which is explained by the fact 
that matrices with continuous distributions are almost surely non-singular. 
In particular, this result does not hold for discrete distributions.

Some related results which apply for discrete distributions are due to Tao and Vu.
Theorem~1.14 in \cite{TV universality edge} states that for every $\d>0$ and $1 \le k \le n$, one has
\begin{equation}							\label{TV spacing}
\P \Big\{ \l_{k+1}(H) - \l_k(H) \le n^{-\frac{1}{2}-\d} \Big\} \le n^{-c(\d)}.
\end{equation}
This result does not assume a continuous distribution of the entries of $H$, 
just appropriate (exponential) moment assumptions. 
In particular, the eigenvalue gaps $\l_{k+1}(H) - \l_k(H)$ are of the order at least $n^{-\frac{1}{2}-\d}$
with high probability. This order is optimal up to $\d$ in the exponent, but the polynomial probability bound $n^{-c(\d)}$
is not. Furthermore, \eqref{delocalization subgaussian} and \eqref{TV spacing} are results of somewhat different nature: 
$\eqref{delocalization subgaussian}$ establishes absolute delocalization of eigenvalues with respect to a given point $z$, 
while $\eqref{TV spacing}$ gives a relative delocalization with respect to the neighboring eigenvalues. 

Finally, recent universality results due to Tao and Vu \cite{TV universality, TV universality edge}
allow to compare the distribution of $\l_k(H)$ to the distribution of $\l_k(G)$ where $G$ is a 
symmetric matrix with independent $N(0,1)$ entries. These results also apply for matrices $H$
with discrete distributions, although one has to assume that the first few moments (such as three or four) of the entries 
of $H$ and of $G$ are equal (so it does not seem that this approach can be used for symmetric Bernoulli
matrices). Also, such comparisons come at a cost of 
a polynomial, rather than exponential, probability error:
\begin{multline}
\P \Big\{ \min_k |\l_k(G)| \le \e n^{-1/2} - n^{-c-1/2} \Big\} - O(n^{-c}) \\
\le \P \Big\{ \min_k |\l_k(H)| \le \e n^{-1/2} \Big\} \\
\le \P \Big\{ \min_k |\l_k(G)| \le \e n^{-1/2} + n^{-c-1/2} \Big\} + O(n^{-c}).
\end{multline}
(See Corollary~24 in \cite{TV universality} and its proof.)
 
\medskip

\begin{remark}
  After the results of this paper had been obtained, the author was informed of an independent work by Nguyen \cite{Nguyen}, 
  which improved Costello-Tao-Vu's singularity probability bound  \eqref{CTV} for symmetric Bernoulli matrices to
  $$
  \P \big\{ H \text{ is singular} \big\} = O(n^{-M})
  $$
  for every $M>0$, where a constant implicit in $O(\cdot)$ depends only on $M$.
  The even more recent work by Nguyen \cite{Nguyen small}, which was announced a few days after the current paper 
  had been posted, demonstrated that for every $M>0$ there exists $K>0$ such that 
  $$
  \P \Big\{ \min_k |\l_k(H)| \le n^{-K} \Big\} \le n^{-M}.
  $$
  While Nguyen's results give weaker conclusions than the results in this paper,
  they hold under somewhat weaker conditions on the distribution than \H (for
  example, the entries of H do not to have mean zero); see \cite{Nguyen small} for precise statements.
\end{remark}

\begin{remark}[Optimality]
  Although the magnitude of the gap $n^{-1/2}$ in Theorem~\ref{delocalization} is optimal, 
  the form of \eqref{iid lowest} and \eqref{ESY} suggests that the exponent $1/9$ 
  is not optimal. Indeed, our argument automatically yields $\e^{1/8 + \d}$ for every $\d>0$
  (with constants $C$, $c$ depending also on $\d$). Some further improvement of the exponent may be possible
  with a more accurate argument, but the technique of this paper would still not reach the optimal exponent $1$
  (in particular, due to losses in decoupling). Furthermore, we conjecture that the singularity 
  probability term $2e^{-n^c}$ in \eqref{eq delocalization subgaussian} may be improved to $2e^{-cn}$.
\end{remark}

\subsection{Four moments}

Even without subgaussian assumption \eqref{subgaussian} on the entries of $H$, 
the bound on the spectral norm $\|H\| = \max_k |\l_k(H)|$ can be removed
from \eqref{eq delocalization}, however this will lead to a weaker probability bound than 
in Theorem~\ref{delocalization subgaussian}:

\begin{theorem}[Four moments]			\label{delocalization 4}
  Let $H$ be an $n \times n$ symmetric random matrix satisfying \H,
  whose off-diagonal entries have finite fourth moment $M_4^4$, and 
  whose diagonal entries satisfy $|h_{ii}| \le K \sqrt{n}$ for some $K$.
  For every $p>0$ there exist $n_0, \e>0$ that depend only on the fourth moment of entries, $K$ and $p$,
  and such that for all $n \ge n_0$ one has
  $$
  \P \Big\{ \min_k |\l_k(H) - z| \le \e n^{-1/2} \Big\} \le p.
  $$
\end{theorem}

To see how this result follows from Theorem~\ref{delocalization}, note that a result of Latala
implies a required bound on the spectral norm. Indeed, Lemma~\ref{norm 4} and Markov's inequality
yield $\|H\| = \max_k |\l_k(H)| \le (CM_4 + K) \sqrt{n}$ with high probability. 
Using this together with \eqref{delocalization} implies Theorem~\ref{delocalization 4}.

\medskip

An immediate consequence of Theorem~\ref{delocalization 4} is that such matrices $H$ 
are asymptotically almost surely non-singular:
$$
\P \big\{ H \text{ is singular} \big\} \le p_n(M_4, K) \to 0
\quad \text{as } n \to \infty.
$$

Like Theorem~\ref{delocalization subgaussian}, Theorem~\ref{delocalization 4} also 
establishes the delocalization of eigenvalues on the optimal scale $n^{-1/2}$
and the bounds on the resolvent \eqref{resolvent whp}, on the norm of the inverse and 
on the condition number \eqref{condition number} -- 
all these hold under just the fourth moment assumption as in Theorem~\ref{delocalization 4}.

\subsection{Overview of the argument}

\paragraph{Decomposition into compressible and incompressible vectors.}
Let us explain the heuristics of the proof of Theorem~\ref{delocalization}. 
Consider the matrix $A = H-zI$. 
Note that $\min_k |\l_k(H)-z| = \min_k |\l_k(A)| = \min_{x \in S^{n-1}} \|Ax\|_2$
where $S^{n-1}$ denotes the Euclidean sphere in $\R^n$. So our task is to bound above the 
probability
$$
\P \Big\{ \min_{x \in S^{n-1}} \|Ax\|_2 \le \e n^{-1/2} \Big\}.
$$
In other words, we need to prove the lower bound $\|Ax\|_2 \gtrsim n^{-1/2}$ 
uniformly for all vectors $x \in S^{n-1}$, and with high probability.

Our starting point is the method developed in \cite{RV square} for a similar
invertibility problem for matrices $A$ with all independent entries, see also \cite{RV ICM}.
We decompose the sphere $S^{n-1} = \Comp \cup \Incomp$ into the classes of 
compressible and incompressible vectors. A vector $x$ is in $\Comp$ if
$x$ is within distance, say, $0.1$ from the set of vectors of support $0.1 n$. 
We seek to establish invertibility of $A$ separately for the two classes,
our goal being
\begin{equation}										\label{informal goal}
\min_{x \in \Comp} \|Ax\|_2 \gtrsim n^{1/2}, 
\quad \min_{x \in \Incomp} \|Ax\|_2 \gtrsim n^{-1/2}.
\end{equation}
(The first estimate is even stronger than we need.)
Each of the two classes, compressible and incompressible, has its own advantages.

\paragraph{Invertibility for compressible vectors.}
The class $\Comp$ has small metric entropy, which makes it amenable to covering arguments. 
This essentially reduces the invertibility problem for $\Comp$ to proving the lower bound
$\|Ax\|_2 \gtrsim n^{1/2}$ with high probability for one (arbitrary) vector $x \in \Comp$.
If $A$ had all independent entries (as in \cite{RV square}) then we could express $\|Ax\|_2^2$ as a sum of 
independent random variables $\sum_{k=1}^n \< A_k, x\> ^2$ where $A_k$ denote the 
rows of $A$, and finish by showing that each $\< A_k, x\> $ is unlikely to be $o(1)$. 
But in our case, $A$ is symmetric, so $A_k$ are not independent. 
Nevertheless, we can extract from $A$ a minor $G$ with all independent entries. 
To this end, consider a subset $I \subset [n]$ with $|I| = \l n$ where $\l \in (0,1)$ is a small number. 
We decompose
\begin{equation}							\label{decomposition intro}
A = 
\begin{pmatrix} 
  D & G \\ 
  G^* & E 
\end{pmatrix}, \quad
x = \begin{pmatrix} y \\ z \end{pmatrix}
\end{equation}
where $D$ is a $I^c \times I^c$ matrix, $G$ is a $I^c \times I$ matrix, $y \in I^c$, $z \in I$. 
Then $\|Ax\|_2 \ge \|Dy+Gz\|_2$. Conditioning on the entries in $D$ and denoting the fixed vector
$-Dy$ by $v$, we reduced the problem to showing that
\begin{equation}										\label{G invertibility}
\|Ax\|_2 \ge \|Gz - v \|_2 \gtrsim n^{1/2} \quad \text{with high probability}.
\end{equation}
Now $G$ is a matrix with all independent entries, so the previous reasoning 
yields \eqref{G invertibility} with probability at least $1 - 2e^{-cn}$.
This establishes the first part of our goal \eqref{informal goal}, i.e. 
the good invertibility of $A$ on the class of compressible vectors.

\paragraph{Concentration of quadratic forms.}
The second part of our goal \eqref{informal goal} is more difficult. 
A very general observation from \cite{RV square} reduces the 
invertibility problem for incompressible vectors to a {\em distance problem}
for a random vector and a random hyperplane (Section~\ref{s: distance}).
Specifically, we need to show that 
\begin{equation}										\label{dist goal}
\dist(X_1,H_1) \gtrsim 1	\quad \text{with high probability},
\end{equation}
where $X_1$ denotes the first column of $A$ and $H_1$ denotes the span of the other $n-1$ columns. 
An elementary observation (Proposition~\ref{dist via quadratic}) is that 
$$
\dist(A_1, H_1) = \frac{\big|\< B^{-1} Z, Z\> - a_{11} \big|}{\sqrt{1 + \|B^{-1}Z\|_2^2}}, 
\quad \text{where }
A = 
\begin{pmatrix} 
  a_{11} & Z \\ 
  Z^* & B 
\end{pmatrix}. 
$$
Obviously the random vector $Z \in \R^{n-1}$ and the $(n-1) \times (n-1)$ symmetric random matrix $B$ 
are independent, and $B$ has the same structure as $A$ (its above-diagonal entries are independent). 
So lifting the problem back into dimension $n$, we arrive at the following problem for {\em quadratic forms}.
Let $X$ be a random vector in $\R^n$ with iid coordinates with mean zero and bounded fourth moment. 
Show that for every fixed $u \in \R$,
\begin{equation}										\label{quadratic goal}
\big| \< A^{-1}X, X\> - u \big| \gtrsim \|A^{-1}\|_\HS	\quad \text{with high probability},
\end{equation}
where $\|\cdot\|_\HS$ denotes the Hilbert-Schmidt norm. In other words, we need to show that
the distribution of the quadratic form $\< A^{-1}X, X\> $ is spread on the real line. 

The spread of a general random variable $S$ is measured by the {\em L\'evy concentration function} 
$$
\LL(S,\e) := \sup_{u \in \R} \P \big\{ |S-u| \le \e \big\}, \quad \e \ge 0.
$$
So our problem becomes to estimate L\'evy concentration function of quadratic forms of the type
$\< A^{-1}X, X\> $ where $A$ is a symmetric random matrix, and $X$ is an independent random vector
with iid coordinates. 

\paragraph{Littlewood-Offord theory.}
A decoupling argument allows one to replace $\< A^{-1} X, X\> $ by the bilinear form $\< A^{-1} Y, X\> $
where $Y$ is an independent copy of $X$. (This is an ideal situation; 
a realistic decoupling argument will incur some losses which we won't discuss here, see Section~\ref{s: decoupling}.)  
Using that $\E \|A^{-1}Y\|_2^2 = \|A^{-1}\|_\HS^2$, we reduce the problem 
to showing that for every $u \in \R$ one has 
\begin{equation}							\label{x0 def}
\big| \< x_0, X\> - u \big| \gtrsim 1	\quad \text{with high probability}, \quad \text{where } x_0 = \frac{A^{-1}Y}{\|A^{-1}Y\|_2}.
\end{equation}
By conditioning on $A$ and $X$ we can consider $x_0$ as a fixed vector. The product 
$$
S := \< x_0, X\> = \sum_{k=1}^n x_0(k) X(k)
$$
is a sum of independent random variables. So our problem reduces to estimating L\'evy concentration function 
for general sums of independent random variables with given coefficients $x_0(k)$.

It turns out that the concentration function depends not only on the magnitude of the coefficients $x_0(k)$, 
but also on their {\em additive structure}. A vector $x_0$ with less `commensurate' coefficients tends to produce 
better estimates for $\LL(S,\e)$.
Many researchers including Littlewood, Offord, Erd\"os, Moser, S\'ark\"ozi, Szemer\'edi and Halasz produced
initial findings of this type; Kahn, Koml\'os and Szemer\'edi \cite{KKS} found applications to the invertibility problem 
for random matrices. Recently this phenomenon was termed the (inverse) {\em Littlewood-Offord theory}
by Tao and Vu \cite{TV Annals}. They  initiated a systematic study of the effect the additive structure of the 
coefficient vector $x_0$ has on the concentration function; see a general 
discussion in \cite{TV Bulletin, RV ICM} with a view toward random matrix theory.

In \cite{RV square, RV rectangular}, Rudelson and the author of this paper proposed 
to quantify the amount of additive structure of a vector $x \in S^{n-1}$
by the {\em least common denominator (LCD)}; the version of LCD we use here (due to Rudelson) is
\begin{equation}							\label{LCD intro}
D(x) = \inf \Big\{ \theta>0: \, \dist(\theta x, \Z^n) \lesssim \sqrt{\log_+ \theta} \Big\}.
\end{equation}
The larger $D(x)$, the less structure $x$ has, the smaller $\LL(S,\e)$ is expected to be.
Indeed, a variant of the Littlewood-Offord theory developed in \cite{RV square, RV rectangular} states that
\begin{equation}							\label{LO}
\LL(S,\e) \lesssim \e + \frac{1}{D(x_0)}, \quad \e \ge 0.
\end{equation}
The actual, more accurate, definition of LCD and the precise statement of \eqref{LO} is given in Section~\ref{s: sbp via lcd}.

\paragraph{Additive structure.}
In order to use Littlewood-Offord theory, one has to show that $D(x_0)$ is large for the vector $x_0$ in \eqref{x0 def}.
This is the main difficulty in this paper, coming from the symmetry restrictions in the matrix $A$.
We believe that the action of $A^{-1}$ on an (arbitrary) vector $Y$ should make the random vector $x_0$ completely unstructured, 
so it is plausible that $D(x_0) \ge e^{cn}$ with high probability, where $c>0$ is a constant.
If so, the singularity probability term in \eqref{eq delocalization} would improve to $e^{-cn}$. 
Unfortunately, we can not even prove that $D(x_0) \ge e^{c\sqrt{n}}$. 

The main losses occur in the process of decoupling and conditioning,
which is performed to reduce the symmetric matrix $A$ to a matrix with all independent entries. 
In order to resist such losses, we propose in this paper to work with an alternative 
(but essentially equivalent) robust version of LCD which we call the {\em regularized LCD}. 
It is designed to capture the most unstructured part of $x$ of a given size. 
So, for a parameter $\l \in (0,1)$, we consider
\begin{equation}							\label{reg LCD intro}
\Dhat(x, \l) = \max \Big\{ D \big(x_I/\|x_I\|_2\big) : \, I \subseteq [n], \, |I| = \lceil \l n \rceil \Big\}
\end{equation}
where $x_I \in \R^I$ denotes the restriction of vector $x$ onto the subset $I$.
The actual, more accurate, definition of regularized LCD is given in Section~\ref{s: reg LCD}.

On the one hand, if $\Dhat(x, \l)$ is large, then $x$ has some unstructured part $x_I$,
so we can still apply the linear Littlewood-Offord theory (restricted to $I$) to produce good bounds on the
L\'evy concentration function for linear forms (Proposition~\ref{sbp via reg lcd}), 
and extend this for quadratic forms by decoupling. 
On the other hand, if $\Dhat(x,\l)$ is small, then not only $x_I$ but {\em all} restrictions of $x$ 
onto arbitrary $\lceil \l n \rceil$ coordinates are nicely structured, so in fact the entire $x$ is highly structured. 
This yields a good control of the metric entropy of the set of vectors with small $\Dhat(x, \l)$.
Ultimately, this approach (explained in more detail below) 
leads us to the desired {\em structure theorem}, which states that for $\l \ge n^{-c}$, one has
\begin{equation}										\label{structure goal}
\Dhat(x_0,\l) \gtrsim n^{c/\l} \quad \text{with high probability}.
\end{equation}
See Theorem~\ref{structure} for the actual statement.
In other words, the structure theorem that the regularized LCD is larger than any polynomial in $n$.
As we explained, this estimate is then used in combination with the Littlewood-Offord theory \eqref{LO}
to deduce estimate \eqref{quadratic goal} for quadratic forms (after optimization in $\l$); 
see Theorem~\ref{sbp quadratic form} for the actual result on concentration of quadratic forms. 
This in turn yields a solution of the distance problem \eqref{dist goal}, see Corollary~\ref{distance}. Ultimately, 
this solves the second part of invertibility problem \eqref{informal goal}, i.e. for the incompressible vectors,
and completes the proof of Theorem~\ref{delocalization}.

\paragraph{The structure theorem.}
The proof of structure theorem \eqref{structure goal} is the main technical ingredient of the paper.
We shall explain heuristics of this argument in some more detail here. 
Let us condition on the independent vector $Y$ in \eqref{x0 def}. 

By definition of $x_0$, the vector $Ax_0$ is co-linear with the fixed vector $Y$, so (apart from the normalization issue, which we ignore now) 
we can assume that $Ax_0$ equals some fixed vector $u \in \R^n$. Then structure theorem \eqref{structure goal}
will follow if we can show that, with high probability, 
all vectors $x \in S^{n-1}$ with $\Dhat(x,\l) \ll n^{c/\l}$ satisfy $Ax \ne u$.

To this end, fix some value $D  \ll n^{c/\l}$ and consider the level set
$$
S_D = \big\{ x \in S^{n-1}:\; \Dhat(x,\l) \sim D \big\}.
$$
Our goal is to show that, with high probability, $Ax \ne u$ for all $x \in S_D$. 
This will be done by a covering argument.

First we show an individual estimate, that for an arbitrary given $x \in S_D$, $Ax \ne u$ with high probability. 
So let us fix $x \in S_D$ and assume that $Ax=u$. We choose the most unstructured subset of indices $I$ of $x$, 
i.e. let $I$ be the maximizing set in definition \eqref{reg LCD intro}
of the regularized LCD. The decomposition $[n] = I^c \cup I$ induces the decomposition of matrix $A$
we considered earlier in \eqref{decomposition intro}. Conditioning on the minor $D$, we estimate
$$
0 = \|Ax-u\|_2 \ge \|Gz-v\|_2 = \sum_{k \in I^c} \big( \< G_k, x_I \> - v_k \big)^2
$$
where $v = (v_1,\ldots,v_n)$ denotes some fixed vector (which depends on $u$ the entries of $D$, which are now fixed), 
and $G_k$ denote the rows of the minor $G$.
It follows that $\< G_k, x_I \> - v_k = 0$ for all $k \in I^c$. 
Since $G$ has independent entries,
the probability of these equalities can be estimated using a Littlewood-Offord estimate \eqref{LO} as 
$$
\P \Big\{ \< G_k, x_I \> - v_k = 0 \Big\} \lesssim \frac{1}{D(x_I)} \sim \frac{1}{\Dhat(x,\l)} \sim \frac{1}{D},
\quad k \in I^c.
$$
Therefore, by independence we have
\begin{equation}							\label{Ax=u}
\P \big\{ Ax=u \big\} \lesssim \Big(\frac{1}{D}\Big)^{|I^c|} 
= \Big(\frac{1}{D}\Big)^{n - \l n} 
\quad \text{for all } x \in S_D.
\end{equation}

On the other hand, the level set $S_D$ has small metric entropy. To see this, first consider the level set of the 
usual LCD in \eqref{LCD intro}:
$$
T_D = \big\{ x \in S^{n-1}:\; D(x) \sim D \big\}.
$$
Since the number of integer points in a Euclidean ball of radius $D$ in $\R^n$ is about $(D/\sqrt{n})^n$, 
the definition of LCD implies that there exists an $\b$-net $\MM$ of $T_D$ in the Euclidean metric with 
$$
\b \sim \frac{\sqrt{\log D}}{D}, \quad |\MM| \lesssim \Big(\frac{D}{\sqrt{n}}\Big)^n.
$$
Now consider an arbitrary $x \in S_D$. By definition of the regularized LCD, the  
restriction $x_I$ of any set $I$ of $\l n$ coordinates has $D(x_I/\|x_I\|_2) \lesssim D$. 
So we can decompose $[n]$ into $1/\l$ sets of indices $I_j$, $|I_j| = \l n$,
and for the restriction of $x$ onto each $I_j$ construct a $\b$-net $\MM_j$ in $\R^{I_j}$
with $|\MM_j| \lesssim (D/\sqrt{\l n})^{\l n}$ as above.
The product of these nets $\MM_j$ obviously forms a $\b/\sqrt{\l}$-net $\NN$ of $S_D$ with 
$$
|\NN| \lesssim \Big( \Big(\frac{D}{\sqrt{\l n}}\Big)^{\l n} \Big)^{1/\l} = \Big(\frac{D}{\sqrt{\l n}}\Big)^n.
$$

Finally, we take a union bound of probability estimates \eqref{Ax=u} over all $x$ in the net $\NN$
of $S_D$. This gives
$$
\P \Big\{ \exists x \in \NN :\; Ax=u \Big\} 
\lesssim \Big(\frac{1}{D}\Big)^{n - \l n} \Big(\frac{D}{\sqrt{\l n}}\Big)^n
= \Big(\frac{D^\l}{\sqrt{\l n}}\Big)^n.
$$
Therefore, if $D \ll (\l n)^{2/\l}$ then the probability bound is exponentially small. 
An approximation argument (using the bound $\|A\| = O(\sqrt{n})$) extends this 
from the net $\NN$ to the entire sub-level set $S_D$, 
and a simple union bound over all $D \ll (\l n)^{2/\l}$ finally yields
$$
\P \Big\{ \exists x \in S^{n-1}, \, \Dhat(x,\l) \ll (\l n)^{2/\l} :\; Ax=u \Big\} 
\lesssim e^{-n}.
$$
As we said, this implies that with (exponentially) large probability, 
$\Dhat(x,\l) \gtrsim (\l n)^{2/\l}$, which is essentially the statement of structure theorem
\eqref{structure goal}.

\section{Notation and initial reductions of the problem}

\subsection{Notation}

Throughout this paper $C, C_1, C_2, c, c_1, c_2, \ldots$ will denote positive constants. When it does not create confusion, the same
letter (say, $C$) may denote different constants in different parts of the proof. The value of the constants
may depend on some natural parameters such as the fourth moment of the entries of $H$, but it will
never depend on the dimension $n$. Whenever possible, we will state which parameters the constant depends on. 

The discrete interval is denoted $[n] = \{1,\ldots,n\}$. The logarithms $\log a$ are natural unless noted otherwise.

$\P \{\EE\} = \P_{X,Y} \{\EE\}$ stands for the probability of an event $\EE$ that depends on the values of random variables, say, $X$ and $Y$.
Similarly, $\E f(X,Y) = \E_{X,Y} f(X,Y)$ stands for the expected value of a certain function $f(X,Y)$ of random variables $X$ and $Y$.

For a vector $x = (x_1,\ldots,x_n) \in \R^n$, the Euclidean norm is $\|x\|_2 = \big( \sum_{k=1}^n |x_k|^2 \big)^{1/2}$ and the sup-norm is 
$\|x\|_\infty = \max_k |x_k|$. The unit Euclidean sphere is $S^{n-1} = \{ x \in \R^n :\; \|x\|_2 = 1 \}$ and the unit Euclidean ball is 
$B_2^n = \{ x \in \R^n :\; \|x\|_2 \le 1 \}$. The Euclidean distance from a point $x \in \R^n$ to a subset $D \subset \R^n$ is denoted
$\dist(x,T) = \inf \{ \|x-t\|_2 :\; t \in T \}$.

Consider a subset $I \subseteq [n]$. The unit Euclidean ball in $\R^I$ is denoted $B_2^I$.
The orthogonal projection in $\R^n$ onto $\R^I$ is denoted $P_I : \R^n \to \R^n$. 
The restriction of a vector $x = (x_1,\ldots,x_n) \in \R^n$ onto the coordinates in $I$ is denoted $x_I$. 
Thus $P_I x$ is a vector in $\R^n$ (with zero coordinates outside $I$), while $x_I = (x_k)_{k \in I}$ is a vector in $\R^I$.

Let $A$ be an $n \times n$ symmetric matrix. 
The eigenvalues of $A$ arranged in a non-decreasing order are denoted $\l_k(A)$.   
The spectral norm of $A$ is 
\begin{equation}							\label{lmax}
\max_k |\l_k(A)| = \max_{x \in S^{n-1}} \|Ax\|_2 = \|A\|.
\end{equation}	
The eigenvalue of the smallest magnitude determines the norm of the inverse:	
\begin{equation}							\label{lmin}
\min_k |\l_k(A)| = \min_{x \in S^{n-1}} \|Ax\|_2 = 1/\|A^{-1}\|. 
\end{equation}							
The transpose of $A$ is denotes $A^*$. The Hilbert-Schmidt norm of $A$ is denoted
$$
\|A\|_\HS = \Big( \sum_{k=1}^n \l_k(A)^2 \Big)^{1/2}.
$$

\subsection{Nets and bounds on the spectral norm}

Consider a compact set $T \in \R^n$ and $\e>0$.
A subset $\NN \subseteq T$ is called an $\e$-net of $T$ if for every point $t \in T$ one has $\dist(t,\NN) \le \e$. 
The minimal cardinality of an $\e$-net of $T$ is called the {\em covering number} of $T$ (for a given $\e$), 
and is denoted $N(T,\e)$.
Equivalently, $N(T,\e)$ is the minimal number of closed Euclidean balls of radii $\e$ and centered in points of $T$, 
whose union covers $T$. 

\begin{remark}[Centering]					\label{centering}
  Suppose $T$ can be covered with $N$ balls of radii $\e$, but their centers are not necessarily in $T$.
  Then enlarging the radii by the factor of $2$, we can place the centers in $T$. So $N(T,2\e) \le N$. 
\end{remark}

\begin{lemma}[See e.g. \cite{V intro rmt}, Lemma~2]				\label{covering sphere}
  For every subset $T \subseteq S^{n-1}$ and every $\e \in (0,1]$, one has
  $$
  N(T,\e) \le (3/\e)^n.
  $$
\end{lemma}

The following known lemma was used to deduce Theorem~\ref{delocalization subgaussian}
for subgaussian matrices from our general result, Theorem~\ref{delocalization}.
 
\begin{lemma}[Spectral norm: subgaussian]			\label{norm subgaussian}
  Let $H$ be a symmetric random matrix as in Theorem~\ref{delocalization subgaussian}.
  Then 
  $$
  \P \Big\{ \|H\| \le (\Const{norm subgaussian}M + K) \sqrt{n} \Big\} \ge 1 - 2e^{-n},
  $$
  where $\Const{norm subgaussian}$ is an absolute constant.
\end{lemma}

\begin{proof}
Let us decompose the matrix as $H=D+B+B^*$ where $D$ is the diagonal part of $H$, and $B$ is the 
above-diagonal part of $H$. Since $\|D\| \le K \sqrt{n}$ by assumption and $\|B\| = \|B^*\|$, we have
$\|H\| \le K \sqrt{n} + 2\|B\|$. Furthermore, since the entries of $B$ on and below the diagonal are zero, 
all $n^2$ entries of $B$ are independent mean zero random variables 
with subgaussian moments bounded by $M$.
Proposition~2.4 of \cite{RV ICM} then implies a required bound on $\|B\|$: 
$$
\P \big\{ \|B\| \le C M\sqrt{n} \big\} \ge 1 - 2e^{-n},
$$
where $C$ is an absolute constant.
This completes the proof. 
\end{proof}

A similar spectral bound holds just under the fourth moment assumption, although
only in expectation.

\begin{lemma}[Spectral norm: four moments]			\label{norm 4}
  Let $H$ be a symmetric random matrix as in Theorem~\ref{delocalization 4}.  
  Then 
  $$
  \E \|H\| \le (\Const{norm 4} M_4 + K) \sqrt{n},
  $$
  Here $\Const{norm 4}$ is an absolute constant.
\end{lemma}

\begin{proof}
We use the same decomposition $H = D + B + B^*$ as in the proof of Lemma~\ref{norm subgaussian}. 
A result of Latala \cite{Latala} implies that $\E \|B\| \le CM_4$ where $C$ is an absolute constant.
Thus
$$
\E\|H\| \le \|D\| + 2\E\|B\| \le (K + 2CM_4) \sqrt{n}.
$$
The lemma is proved.
\end{proof}

\subsection{Initial reductions of the problem}				\label{s: reductions}

We are going to prove Theorem~\ref{delocalization}. 
Without loss of generality, we can assume that $K \ge 1$ by increasing this value. 
Also we can assume that the constant $c$ in this theorem is sufficiently small, 
depending on the value of the fourth moment and on $K$. Consequently, 
we can assume that $n \ge n_0$ where $n_0$ is a sufficiently large number 
that depends on the fourth moment and on $K$. (For $n<n_0$ the probability bound
in \eqref{eq delocalization} will be larger than $1$, which is trivially true.)  
By a similar reasoning, we can assume that $\e \in (0, \e_0)$ for a sufficiently small number $\e_0>0$
which depends on the fourth moment and on $K$.

So we can assume that $K \sqrt{n} \ge \e n^{-1/2}$. 
Therefore, for $|z| > 2 K \sqrt{n}$ the probability in question is automatically zero.
So we can assume that $|z| \le 2 K \sqrt{n}$.

We shall work with the random matrix
$$
A = H - z I.
$$
If $\|H\| = \max_k |\l_k(H)| \le K \sqrt{n}$ as in \eqref{eq delocalization} 
then $\|A\| \le \|H\| + |z| \le 3 K \sqrt{n}$. 
Therefore, the probability of the desired event in \eqref{eq delocalization} 
is bounded above by 
$$
p:= \P \Big\{ \min_k |\l_k(A)| \le \e n^{-1/2} 
  \wedge \EE_K \Big\}
$$
where $\EE_K$ denotes the event
\begin{equation}										\label{norm}
\EE_K = \big\{ \|A\| \le 3K\sqrt{n} \big\}.
\end{equation}

Using \eqref{lmin}, we see that Theorem~\ref{delocalization} would follow if we prove that
\begin{equation}										\label{delocalization goal}
p: = \P \Big\{ \min_{x \in S^{n-1}} \|Ax\|_2 \le \e n^{-1/2} 
  \wedge \EE_K \Big\} \le C \e^{1/9} + 2e^{-n^c}.
\end{equation}

We do this under the following assumptions on the random matrix $A$:

\begin{itemize} 
\item[\A] $A = (a_{ij})$ is an $n \times n$ real symmetric matrix. The above-diagonal entries $a_{ij}$, $i<j$, 
  are independent and identically distributed random variables with 
  \begin{equation}										\label{aij}
  \E a_{ij} = 0, \quad \E a_{ij}^2 = 1, \quad \E a_{ij}^4 \le M_4^4 \quad \text{ for } j > i,
  \end{equation}
  where $M_4$ is some finite number. The diagonal entries $a_{ii}$ are 
  arbitrary fixed numbers.
\end{itemize}

The constants $C$ and $c>0$ in \eqref{delocalization goal} will have to depend only on 
$K$ and $M_4$. 

By a small perturbation of the entries of $A$ (e.g. adding independent normal random variables
with zero means and small variances), we can assume that the distribution of the entries $a_{ij}$ is 
absolutely continuous. In particular, the columns of $A$ are in a general position almost surely. 
So the matrix $A$ as well as all of its square minors are invertible almost surely; 
this allows us to ignore some technicalities that can arise in degenerate cases.

\section{Preliminaries: small ball probabilities, compressible and incompressible vectors}

In this section we recall some preliminary material from \cite{RV square, RV rectangular}.

\subsection{Small ball probabilities, L\'evy concentration function}

\begin{definition}[Small ball probabilities]
  Let $Z$ be a random vector in $\R^n$. 
  The {\em L\'evy concentration function} of $Z$ is defined as 
  $$
  \LL(Z,\e) = \sup_{u \in \R^n} \P \big\{ \|Z-u\|_2 \le \e \big\}.
  $$
\end{definition}

The L\'evy concentration function bounds the {\em small ball probabilities} for $Z$, 
which are the probabilities that $Z$ falls in a Euclidean ball of radius $\e$.

A simple but rather weak bound on L\'evy concentration function follows from Paley-Zygmund inequality. 

\begin{lemma}[\cite{RV rectangular}, Lemma~3.2]       		\label{Levy weak}
  Let $Z$ be a random variable with unit variance and with finite fourth moment, 
  and put $M_4^4 := \E (Z-\E Z)^4$.
  Then for every $\e \in (0,1)$ there exists $p = p(M_4, \e) \in (0,1)$ 
  such that
  $$
  \LL(\xi, \e) \le p. 
  $$
\end{lemma}

There has been a significant interest in bounding L\'evy concentration function for 
sums of independent random variables;
see \cite{RV square, RV rectangular, RV ICM} for discussion.
The following simple but weak bound was essentially  
proved in \cite{RV square}, Lemma~2.6 (up to centering). 

\begin{lemma}[L\'evy concentration function for sums]				\label{sbp sums}
  Let $\xi_1, \ldots, \xi_n$ be independent random variables with 
  unit variances and $\E (\xi_k - \E \xi_k)^4 \le M_4^4$, 
  where $M_4$ is some finite number. 
  Then for every $\e \in (0,1)$ there exists $p = p(M_4, \e) \in (0,1)$ 
  such that the following holds. 
  
  For every vector $x = (x_1,\ldots, x_n) \in S^{n-1}$, the sum 
  $S = \sum_{k=1}^n x_k \xi_k$ satisfies
  $$
  \LL(S, \e) \le p. 
  $$
\end{lemma}

\begin{proof}
Clearly $S$ has unit variance. Furthermore, since
$S - \E S = \sum_{k=1}^n x_k (\xi_k - \E \xi_k)$, an application of Khinchine 
inequality yields
$$
\E (S - \E S)^4 \le C M_4^4, 
$$
where $C$ is an absolute constant (see \cite{RV square}, proof of Lemma~2.6). 
The desired concentration bound then follows from Lemma~\ref{Levy weak} with $Z = S - \E S$.
\end{proof}

The following tensorization lemma can be used to transfer bounds for the
L\'evy concentration function from random variables to random vectors. 
This result follows from \cite{RV square}, Lemma~2.2 with $\xi_k = |x_k - u_k|$, 
where $u = (u_1,\ldots,u_n) \in \R^n$.

\begin{lemma}[Tensorization]					\label{tensorization}
  Let $X = (X_1, \ldots, X_n)$ be a random vector in $\R^n$ with independent 
  coordinates $X_k$. 
  \begin{enumerate}[1.]
    \item Suppose there exists numbers $\e_0 \ge 0$ and $L \ge 0$ such that 
      $$
      \LL(X_k, \e) \le L\e \text{ for all } \e \ge \e_0 \text{ and all } k.
      $$
      Then       
      $$
      \LL(X, \e\sqrt{n}) \le (\Const{tensorization}L\e)^n \text{ for all } \e \ge \e_0,
      $$
      where $\Const{tensorization}$ is an absolute constant. 
    \item Suppose there exists numbers $\e > 0$ and $p \in (0,1)$ such that 
      $$
      \LL(X_k, \e) \le p \text{ for all } k.
      $$
      There exists numbers $\e_1 = \e_1(\e,p) > 0$ and $p_1 = p_1(\e,p) \in (0,1)$
      such that 
      $$
      \LL(X, \e_1 \sqrt{n}) \le p_1^n.
      $$     
 \end{enumerate}
\end{lemma}

\begin{remark}					\label{tensorization remark}
  A useful equivalent form of Lemma~\ref{tensorization} (part 1) is the following one. 
  Suppose there exist numbers $a, b \ge 0$ such that 
  $$
  \LL(X_k, \e) \le a\e + b  \text{ for all } \e \ge 0 \text{ and all } k.
  $$
  Then 
  $$
  \LL(X, \e) \le \big[ \Const{tensorization remark}a\e + b) \big]^n  \text{ for all } \e \ge 0,
  $$   
  where $\Const{tensorization remark}$ is an absolute constant.
\end{remark}

\subsection{Compressible and incompressible vectors}

Let $c_0, c_1 \in (0,1)$ be two numbers. We will choose their values later
as small constants that depend only on the parameters $K$ and $M_4$ 
from \eqref{delocalization goal} and \A, see Remark~\ref{c0 c1} below.

\begin{definition}[\cite{RV rectangular}, Definition~2.4] 
  A vector $x \in \R^n$ is called {\em sparse} if $|\supp(x)| \le c_0 n$.
  A vector $x \in S^{n-1}$ is called {\em compressible} if $x$
  is within Euclidean distance $c_1$ from the set of all sparse vectors.
  A vector $x \in S^{n-1}$ is called {\em incompressible} if it is not compressible.
  
  The sets of compressible and incompressible vectors in $S^{n-1}$
  will be denoted by $\Comp(c_0, c_1)$ and $\Incomp(c_0, c_1)$ respectively.
\end{definition}

The classes of compressible and incompressible vectors each have their own advantages. 
The set of compressible vectors has small covering numbers, which are exponential in $c_0 n$ 
rather than in $n$:

\begin{lemma}[Covering compressible vectors]			\label{covering Comp}
  One has
  $$
  N \big(\Comp(c_0, c_1), 2c_1 \big) \le (9/c_0c_1)^{c_0n}.
  $$
\end{lemma}

\begin{proof}
Let $s = \lfloor c_0 n\rfloor$. 
By Lemma~\ref{covering sphere}, the unit sphere $S^{s-1}$ of $\R^s$ can be covered with 
at most $(3/c_1)^s$ Euclidean balls of radii $c_1$. Therefore, the set $S$ of sparse vectors in $\R^n$
can be covered with at most $\binom{n}{s} (3/c_1)^s$ Euclidean balls of radii $c_1$ centered in $S$. 
Enlarging the radii of these balls we conclude that $\Comp(c_0,c_1)$
can be covered with at most $\binom{n}{s} (3/c_1)^s$ Euclidean balls of radii $2c_1$ centered in $S$.
The conclusion of the lemma follows by estimating $\binom{n}{s} \le (en/s)^s$, 
which is a consequence of Stirling's approximation.
\end{proof}

The set of incompressible vectors have a different advantage. Each incompressible vector $x$
has a set of coordinates of size proportional to $n$, whose magnitudes are 
all of the same order $n^{-1/2}$. We can say that an incompressible vector 
is {\em spread} over this set:

\begin{lemma}[Incompressible vectors are spread, \cite{RV square}, Lemma~3.4]			\label{incomp spread}
  For every $x \in \Incomp(c_0,c_1)$, one has 
  $$
  \frac{c_1}{\sqrt{2n}} \le |x_k| \le \frac{1}{\sqrt{c_0 n}}
  $$
  for at least $\frac{1}{2} c_0 c_1^2 n$ coordinates $x_k$ of $x$. 
\end{lemma}

Since $S^{n-1}$ can be decomposed into two disjoint sets $\Comp(c_0,c_1)$ and $\Incomp(c_0,c_1)$,
the problem of proving \eqref{delocalization goal} reduces to establishing the good invertibility 
of the matrix $A$ on these two classes separately: 
\begin{multline}						\label{split}
\P \Big\{ \min_{x \in S^{n-1}} \|Ax\|_2 \le \e n^{-1/2} \wedge \EE_K \Big\} 
\le \P \Big\{ \inf_{x \in \Comp(c_0,c_1)} \|Ax\|_2 \le \e n^{-1/2} \wedge \EE_K \Big\} \\
+ \P \Big\{ \inf_{x \in \Incomp(c_0,c_1)} \|Ax\|_2 \le \e n^{-1/2} \wedge \EE_K \Big\}.
\end{multline}

\subsection{Invertibility for incompressible vectors via the distance problem}	\label{s: distance}

The first part of the invertibility problem \eqref{split}, for compressible vectors, will be settled
in Section~\ref{s: compressible}. The second part, for incompressible vectors,
quickly reduces to a {\em distance problem} for a random vector and a random hyperplane:

\begin{lemma}[Invertibility via distance, \cite{RV square}, Lemma~3.5] 		 \label{l: via distance}
  Let $A$ be any $n \times n$ random matrix. 
  Let $A_1,\ldots,A_n$ denote the columns of $A$, and 
  let $H_k$ denote the span of all columns except the $k$-th.
  Then for every $c_0, c_1 \in (0,1)$ and every $\e \ge 0$, one has
  \begin{equation}							\label{eq via distance}
  \P \Big\{ \inf_{x \in \Incomp(c_0,c_1)} \|A x\|_2 \le \e n^{-1/2} \Big\}
    \le \frac{1}{c_0 n} \sum_{k=1}^n
          \P \big\{ \dist( A_k, H_k) \le c_1^{-1} \e \big\}.
  \end{equation}
\end{lemma}

This reduces our task to finding a lower bound for $\dist(A_k,H_k)$. This distance problem
will be studied in the second half of the paper following Section~\ref{s: compressible}.

\begin{remark}				\label{adding EEK}
  Since the distribution of a random matrix $A$ is completely general in Lemma~\ref{l: via distance}, 
  by conditioning on $\EE_K$ we can replace the conclusion \eqref{eq via distance} by 
  $$
  \P \Big\{ \inf_{x \in \Incomp(c_0,c_1)} \|A x\|_2 \le \e n^{-1/2} \wedge \EE_K \Big\}
    \le \frac{1}{c_0 n} \sum_{k=1}^n
          \P \big\{ \dist( A_k, H_k) \le c_1^{-1} \e \wedge \EE_K \big\}.
  $$  
\end{remark}

\section{Invertibility for compressible vectors}					\label{s: compressible}

In this section we establish a uniform lower bound for $\|Ax\|_2$ on the set 
of compressible vectors $x$. This solves the first part of the invertibility 
problem in \eqref{split}. 

\subsection{Small ball probabilities for $Ax$}

We shall first find a lower bound for $\|Ax\|_2$ for a fixed vector $x$. 
We start with a very general estimate. It will be improved later to a finer result, Proposition~\ref{sbp Ax via LCD},
which will take into account the additive structure of $x$. 

\begin{proposition}[Small ball probabilities for $Ax$]			\label{sbp Ax}
  Let $A$ be a random matrix which satisfies \A. Then for every $x \in S^{n-1}$, one has 
  $$
  \LL(Ax, \const{sbp Ax} \sqrt{n}) \le 2e^{-\const{sbp Ax} n}.
  $$
  Here $\const{sbp Ax}>0$ depends only on the parameter $M_4$ 
  from assumptions \eqref{aij}.
\end{proposition}

\begin{proof}
Our goal is to prove that, for an arbitrary fixed vector $u \in \R^n$, one has
$$
\P \big\{ \|Ax-u\|_2^2 \le \const{sbp Ax}^2 n \big\} \le 2e^{-\const{sbp Ax} n}.
$$

Let us decompose the set of indices $[n]$ into two sets of roughly equal 
sizes, $\{1,\ldots, n_0\}$ and $\{n_0+1,\ldots, n\}$ where $n_0 = \lceil n/2 \rceil$.
This induces the decomposition of the matrix $A$ and both vectors in question,
which we denote  
$$
A = 
\begin{pmatrix} 
  D & G \\ 
  G^* & E 
\end{pmatrix}, \quad
x = \begin{pmatrix} y \\ z \end{pmatrix}, \quad
u = \begin{pmatrix} v \\ w \end{pmatrix}.
$$
This way, we express
\begin{equation}										\label{Ax-u split}
\|Ax-u\|_2^2 = \|Dy + Gz - v\|_2^2 + \|G^*y + Ez - w\|_2^2.
\end{equation}
We shall estimate the two terms separately, using that each of the matrices $G$ and $G^*$
has independent entries. 

We condition on an arbitrary realization of $D$ and $E$, and we express
$$
\|Dy + Gz - v\|_2^2 = \sum_{j=1}^n \big( \< G_j, z \> - d_j \big)^2
$$
where $G_j$ denote the rows of $G$ and $d_j$ denote the coordinates
of the fixed vector $Dy-v$.
For each $j$, we observe that $\< G_j, z \> = \sum_{i=n_0+1}^n a_{ij} x_i$ 
is a sum of independent random variables, and $\sum_{i=n_0+1}^n x_i^2 = \|z\|_2^2$.
Therefore Lemma~\ref{sbp sums} can be applied to control the small ball probabilities as
$$
\LL \Big( \big\langle G_j, \frac{z}{\|z\|_2} \big\rangle, \frac{1}{2} \Big) \le c_3 \in (0,1)
$$
where $c_3$ depends only on the parameter $M_4$ from assumptions \eqref{aij}.

Further, we apply Tensorization Lemma~\ref{tensorization} (part~2) for the vector $Gz/\|z\|_2$ 
with coordinates $\< G_j, z/\|z\|_2 \> $, $j=1, \ldots, n_0$. It follows that there exist numbers
$c_2>0$ and $c_3 \in (0,1)$ that depend only on $M_4$ and  
such that 
$$
\LL(Gz, c_2 \|z\|_2 \sqrt{n_0}) = \LL(Gz/\|z\|_2, c_2 \sqrt{n_0}) \le c_3^{n_0}.
$$ 
Since $Dy-v$ is a fixed vector, this implies 
\begin{equation}										\label{Gz}
\P \big\{ \|Dy + Gz - v\|_2^2 \le c_2^2 \, \|z\|_2^2 \, n_0 \big\} \le c_3^{n_0}. 
\end{equation}
Since this holds conditionally on an arbitrary realization of $D$, $E$, it also holds
unconditionally. 

By a similar argument we obtain that 
\begin{equation}										\label{G-star y}
\P \big\{ \|G^*y + Ez - w\|_2^2 \le c_2^2 \, \|y\|_2^2 \, (n-n_0) \big\} \le c_3^{n-n_0}. 
\end{equation}
Since $n_0 \ge n/2$ and $n-n_0 \ge n/3$ and $\|y\|_2^2 + \|z\|_2^2 = \|x\|_2^2 = 1$, we have 
$c_2^2 \, \|z\|_2^2 \, n_0 + c_2^2 \, \|y\|_2^2 \, (n-n_0) > \frac{1}{3} c_2^2 n$. 
Therefore, by \eqref{Ax-u split}, the inequality $\|Ax-u\|_2^2 \le \frac{1}{3} c_2^2 n$ implies that 
either the event in \eqref{Gz} holds, or the event in \eqref{G-star y} holds, or both. 
By the union bound, we conclude that 
$$
\P \Big\{ \|Ax-u\|_2^2 \le \frac{1}{3} c_2^2 n \Big\} \le c_3^{n_0} + c_3^{n-n_0} \le 2 c_3^{n/3}.  
$$
This completes the proof.
\end{proof}

\subsection{Small ball probabilities for $Ax$ uniformly over compressible $x$}

An approximation argument allows us to extend Proposition~\ref{sbp Ax} to a uniform invertibility bound 
on the set of compressible vectors $x$ uniformly. The following result gives a satisfactory 
answer for the first part of the invertibility problem in \eqref{split}, i.e. for the set of compressible vectors. 
We shall state a somewhat stronger result that is needed at this moment; the stronger
form will be useful later in the proof of Lemma~\ref{inverse incompressible}.

\begin{proposition}[Small ball probabilities for compressible vectors]		\label{sbp comp}
  Let $A$ be an $n \times n$ random matrix which satisfies \A, and let $K \ge 1$.
  There exist $c_0, c_1, \const{sbp comp} \in (0,1)$ that depend only on 
  $K$ and $M_4$ from assumptions \eqref{norm}, \eqref{aij},
  and such that the following holds. For every $u \in \R^n$, one has 
  \begin{equation}										\label{eq sbp comp}
  \P \Big\{ \inf_{\frac{x}{\|x\|_2} \in \Comp(c_0,c_1)} \|Ax-u\|_2 / \|x\|_2 \le \const{sbp comp} \sqrt{n} \wedge \EE_K \Big\} 
  \le 2e^{-\const{sbp comp} n}.
  \end{equation}
\end{proposition}

\begin{proof}
Let us fix some small values of $c_0$, $c_1$ and $\const{sbp comp}$; the precise choice will be made shortly. 
According to Lemma~\ref{covering Comp}, there exists a $(2c_1)$-net $\NN$ 
of the set $\Comp(c_0,c_1)$ such that 
\begin{equation}										\label{NN size}
|\NN| \le (9 / c_0 c_1)^{c_0 n}. 
\end{equation}
Let $\EE$ denote the event in the left hand side of \eqref{eq sbp comp} whose probability we would like to bound. 
Assume that $\EE$ holds. Then there exist vectors 
$x_0 := x/\|x\|_2 \in \Comp(c_0,c_1)$ and $u_0 := u/\|x\|_2 \in \Span(u)$ such that 
\begin{equation}										\label{Ax0-u0}
\|Ax_0 - u_0\|_2 \le \const{sbp comp} \sqrt{n}.
\end{equation}
By the definition of $\NN$, there exists $y_0 \in \NN$ such that 
\begin{equation}										\label{x0-y0}
\|x_0 - y_0\|_2 \le 2c_1.
\end{equation}

On the one hand, by definition \eqref{norm} of event $\EE_K$, we have
\begin{equation}										\label{Ay0}
\|Ay_0\|_2 \le \|A\| \le 3K \sqrt{n}.
\end{equation}
On the other hand, it follows from \eqref{Ax0-u0} and \eqref{x0-y0} that 
\begin{equation}					\label{Ay0-u0}
\|Ay_0-u_0\|_2
  \le \|A\| \|x_0-y_0\|_2 + \|Ax_0 - u_0\|_2
  \le 6c_1 K \sqrt{n} + \const{sbp comp} \sqrt{n}.  
\end{equation}
This and \eqref{Ay0} yield that
$$
\|u_0\|_2 \le 3K \sqrt{n} + 6c_1 K \sqrt{n} + \const{sbp comp} \sqrt{n} \le 10 K \sqrt{n}.
$$
So, we see that
$$
u_0 \in \Span(u) \cap 10 K \sqrt{n} B_2^n =: E.
$$

Let $\MM$ be some fixed $(c_1 K \sqrt{n})$-net of the interval $E$, such that 
\begin{equation}										\label{MM size}
|\MM| \le \frac{20 K \sqrt{n}}{c_1 K \sqrt{n}} = \frac{20}{c_1}.
\end{equation}
Let us choose a vector $v_0 \in \MM$ such that $\|u_0-v_0\|_2 \le c_1 K \sqrt{n}$. 
It follows from \eqref{Ay0-u0} that 
$$
\|Ay_0-v_0\|_2 \le 6c_1 K \sqrt{n} + \const{sbp comp} \sqrt{n} + c_1 K \sqrt{n}
\le (7c_1 K + \const{sbp comp}) \sqrt{n}.
$$
Choose values of $c_1, \const{sbp comp} \in (0,1)$ so that $7c_1 K + \const{sbp comp} \le \const{sbp Ax}$, 
where $\const{sbp Ax}$ is the constant from Proposition~\ref{sbp Ax}. 

Summarizing, we have shown that the event $\EE$ implies the existence 
of vectors $y_0 \in \NN$ and $v_0 \in \MM$ such that
$\|Ay_0-v_0\|_2 \le \const{sbp Ax} \sqrt{n}$. Taking the union bound over $\NN$ and $\MM$, 
we conclude that 
$$
\P(\EE) \le |\NN| \cdot |\MM| \max_{y_0 \in \NN, \, v_0 \in \MM} 
  \P \big\{ \|Ay_0-v_0\|_2 \le \const{sbp Ax} \sqrt{n} \big\}.
$$
Applying Proposition~\ref{sbp Ax} and using the estimates \eqref{NN size}, \eqref{MM size} 
on the cardinalities of the nets, we obtain
$$
\P(\EE) \le \Big( \frac{9}{c_0 c_1} \Big)^{c_0 n} \cdot \frac{20}{c_1} \cdot 2e^{-\const{sbp Ax} n}.
$$
Choosing $c_0>0$ small enough depending on $c_1$ and $\const{sbp Ax}$, we can 
ensure that 
$$
\P(\EE) \le 2 e^{-\const{sbp Ax}n/2}
$$
as required. This completes the proof.
\end{proof}

As an immediate consequence of Proposition~\ref{sbp comp}, we obtain 
a very good bound for the first half of the invertibility problem in \eqref{split}. 
Indeed, since $\e n^{-1/2} \le \const{sbp comp} \sqrt{n}$, we have 
\begin{equation}							\label{compressible solved}
\P \Big\{ \inf_{x \in \Comp(c_0,c_1)} \|Ax\|_2 \le \e n^{-1/2} \wedge \EE_K \Big\}
\le 2 e^{-\const{sbp comp} n}.
\end{equation}

\begin{remark}[Fixing $c_0$, $c_1$]			\label{c0 c1}
  At this point we fix some values $c_0 = c_0(K, M_4)$ and $c_1 = c_1(K, M_4)$ 
  satisfying Proposition~\ref{sbp comp}, for the rest of the argument.
\end{remark}

\section{Distance problem via small ball probabilities for quadratic forms}		\label{s: distance via sbp}

The second part of the invertibility problem in \eqref{split} -- the one for for incompressible vectors --
is more difficult. Recall that Lemma~\ref{l: via distance} reduces the invertibility problem to 
the distance problem, namely to an upper bound on the probability  
$$
\P \big\{ \dist(A_1, H_1) \le \e \big\}
$$
where $A_1$ is the first column of $A$ and $H_1$ is the span of the other columns. 
(By a permutation of the indices in $[n]$, the same bound would hold for all $\dist(A_k, H_k)$
as required in Lemma~\ref{l: via distance}.)

The following proposition reduces the distance problem to the small ball probability 
for quadratic forms of random variables: 

\begin{proposition}[Distance problems via quadratic forms]			\label{dist via quadratic}
  Let $A = (a_{ij})$ be an arbitrary $n \times n$ matrix.
  Let $A_1$ denote the first column of $A$ and $H_1$ denote the span of the other columns. 
  Furthermore, let $B$ denote the $(n-1) \times (n-1)$ minor of $A$ obtained by removing the first row and
  the first column from $A$, and let $X \in \R^{n-1}$ denote the first column of $A$ with the first entry removed.
  Then
  $$
  \dist(A_1, H_1) = \frac{\big|\< B^{-1} X, X\> - a_{11} \big|}{\sqrt{1 + \|B^{-1}X\|_2^2}}.  
  $$
  \end{proposition}

\begin{proof}
Let $h \in S^{n-1}$ denote a normal to the hyperplane $H_1$; choose the sign of the normal arbitrarily. 
We decompose 
$$
A = 
\begin{pmatrix}
  a_{11} & X^* \\
  X & B
\end{pmatrix}, \quad 
A_1 = 
\begin{pmatrix}
  a_{11} \\ X
\end{pmatrix}, \quad 
h =
\begin{pmatrix}
h_1 \\ g
\end{pmatrix},
$$
where $h_1 \in \R$ and $g \in \R^{n-1}$.
Then
\begin{equation}							\label{dist split}
\dist(A_1, H_1) = |\< A_1, h\> | = |a_{11} h_1 + \< X, g\> |.
\end{equation}
Since $h$ is orthogonal to the columns of the matrix $\binom{X^*}{B}$, we have
$$
0 = \begin{pmatrix} X^* \\ B \end{pmatrix}^* h = h_1 X + Bg,
$$
so 
\begin{equation}							\label{g}
g = -h_1 B^{-1} X.
\end{equation}
Furthermore, 
$$
1 = \|h\|_2^2 = h_1^2 + \|g\|_2^2 = h_1^2 + h_1^2 \|B^{-1}X\|_2^2.
$$
Hence
\begin{equation}							\label{h1}
h_1^2 = \frac{1}{1 + \|B^{-1}X\|_2^2}.
\end{equation}
So, using \eqref{g} and \eqref{h1}, we can express the distance in \eqref{dist split} as 
$$
\dist(A_1, H_1) = \big| a_{11} h_1 - \< h_1 B^{-1} X, X\> \big|
= \frac{\big|\< B^{-1} X, X\> - a_{11} \big|}{\sqrt{1 + \|B^{-1}X\|_2^2}}.
$$
This completes the proof. 
\end{proof}

\begin{remark}[$A$ versus $B$]				\label{A versus B}
  Let us apply Proposition~\ref{dist via quadratic} to the $n \times n$ random matrix $A$ which 
  satisfies assumptions \A. Recall that $a_{11}$ is a fixed number, so the problem reduces
  to estimating the small ball probabilities for the quadratic form $\< B^{-1} X, X\> $.
  Observe that $X$ is a random vector that is independent of $B$, and whose entries
  satisfy the familiar moment assumptions \eqref{aij}.
  
  The random matrix $B$ has the same structure as $A$ except it is $(n-1) \times (n-1)$
  rather than $n \times n$.
  For this reason, it will be convenient to develop the theory in dimension $n$, 
  that is for the quadratic forms $\< A^{-1} X, X\> $, 
  where $X$ is an independent random vector. 
  At the end, the theory will be applied in dimension $n-1$ for the matrix $B$.
\end{remark}

\section{Small ball probabilities for quadratic forms via additive structure}

In order to produce good bounds (super-polynomial) 
for the small ball probabilities for the quadratic forms $\< A^{-1} X, X\> $,
we will have to take into account the additive structure of the vector $A^{-1} X$. 
Let us first review the corresponding theory for linear forms, which is sometimes called 
the {\em Littlewood-Offord theory}. We will later extend it (by decoupling) to quadratic forms.

\subsection{Small ball probabilities via LCD}		\label{s: sbp via lcd}

The linear Littlewood-Offord theory concerns the small ball probabilities for the sums of the form
$\sum x_k \xi_k$ where $\xi_k$ are identically distributed independent random variables, and 
$x = (x_1,\ldots,x_n) \in S^{n-1}$ is a given coefficient vector. Lemma~\ref{sbp sums} gives a general bound 
on the concentration function, $\LL(S, \e) \le p$. But this bound is too weak -- it produces a fixed probability $p$
for all $\e$, even when $\e$ approaches zero. Finer estimates are not possible for general sums; for example, 
the sum $S = \pm 1 \pm 1$ with random independent signs equals zero with fixed probability $1/2$. 
Nevertheless, one can break the barrier of fixed probability by taking into account the additive structure 
in the coefficient vector $x$. 

The amount of additive structure in $x \in \R^n$ is captured by the {\em least common denominator}
(LCD) of $x$. If the coordinates $x_k = p_k/q_k$ are rational numbers, then 
a suitable measure of additive structure in $x$ is the least denominator $D(x)$ of these ratios, 
which is the common multiple of the integers $q_k$.  
Equivalently, $D(x)$ the smallest number $\theta>0$ such that $\theta x \in \Z^n$.
An extension of this concept for general vectors with real coefficients was developed in 
\cite{RV square, RV rectangular}, see also \cite{RV ICM}; 
the particular form of this concept we shall use here is proposed by M.~Rudelson (unpublished).

\begin{definition}[LCD]
  Let $L \ge 1$. We define the {\em least common denominator (LCD)} of $x \in S^{n-1}$ as
  $$
  D_L(x) = \inf \Big\{ \theta>0: \, \dist(\theta x, \Z^n) < L \sqrt{\log_+ (\theta/L)} \Big\}.
  $$
  If the vector $x$ is considered in $\R^I$ for some subset $I \subseteq [n]$, then in this definition 
  we replace $\Z^n$ by $\Z^I$.
\end{definition}

Clearly, one always has $D_L(x) > L$. A more sensitive but still quite simple bound is the following one: 

\begin{lemma}				\label{usual LCD large}
  For every $x \in S^{n-1}$ and every $L \ge 1$, one has
  $$
  D_L(x) \ge \frac{1}{2\|x\|_\infty}.
  $$
\end{lemma} 

\begin{proof}
Let $\theta := D_L(x)$, and assume that $\theta < \frac{1}{2\|x\|_\infty}$.  
Then $\|\theta x\|_\infty < 1/2$. Therefore, by looking at the coordinates of the vector $\theta x$
one sees that the vector $p \in \Z^n$ that minimizes $\|\theta x - p\|_2$ 
is $p=0$. So 
$$
\dist(\theta x, \Z^n) = \|\theta x\|_2 = \theta.
$$
On the other hand, by the definition of LCD, we have 
$$
\dist(\theta x, \Z^n) \le L \sqrt{\log_+ (\theta/L)}.
$$
However, the inequality $\theta \le L \sqrt{\log_+ (\theta/L)}$ has no solutions in $\theta \ge 0$. 
This contradiction completes the proos. 
\end{proof}

The goal of our variant of Littlewood-Offord theory is to express the small ball probabilities of sums $\LL(S, \e)$ 
in terms of $D(x)$. 
This is done in the following theorem, which is a version of results from \cite{RV square, RV rectangular}; 
this particular simplified form is close to the form put forth by M.~Rudelson (unpublished).

\begin{theorem}[Small ball probabilities via LCD]				\label{sbp via lcd}
  Let $\xi_1,\ldots,\xi_n$ be independent and identically distributed random variables. Assume that 
  there exist numbers $\e_0, p_0, M_1 > 0$ such that $\LL(\xi_k, \e_0) \le 1-p_0$ and $\E |\xi_k| \le M_1$ for all $k$.
  Then there exists $\Const{sbp via lcd}$ which depends only on $\e_0$, $p_0$ and $M_1$, and such that the following holds. 
  Let $x \in S^{n-1}$ and consider the sum 
  $S = \sum_{k=1}^n x_k \xi_k$.
  Then for every $L \ge p_0^{-1/2}$ and $\e \ge 0$ one has
  $$
  \LL(S,\e) \le \Const{sbp via lcd} L \Big( \e + \frac{1}{D_L(x)} \Big).
  $$
\end{theorem}

The proof of Theorem~\ref{sbp via lcd} is based on Esseen's Lemma, see e.g. \cite{TV}, p.~290.
\begin{lemma}[C.-G.~Esseen] 					 \label{Esseen}
  Let $Y$ be a random variable. Then
  $$
  \LL(Y,1) \le \Const{Esseen} \int_{-1}^1 |\phi_Y(\theta)| \, d\theta
  $$
  where $\phi_Y(\theta) = \E \exp (2\pi i \theta Y)$ is the
  characteristic function of $Y$, and $\Const{Esseen}$ is an absolute constant.
\end{lemma}

\begin{proof}[Proof of Theorem~\ref{sbp via lcd}.]
By replacing $\xi_k$ with $\xi_k/\e_0$, we can assume without loss of generality that $\e_0=1$. 
We apply Esseen's Lemma~\ref{Esseen} for $Y = S/\e$. Using independence of $\xi_k$, we obtain
\begin{equation}							\label{LSe product}
\LL(S,\e) \le \Const{Esseen} \int_{-1}^1 \prod_{k=1}^n \Big| \phi \Big( \frac{\theta x_k}{\e} \Big) \Big| \, d\theta,
\end{equation}
where $\phi(t) = \E \exp(2 \pi i t \xi)$ is the characteristic function of $\xi := \xi_1$. 

We proceed with a conditioning argument similar to the ones used in \cite{R, RV square, RV rectangular}.
Let $\xi'$ denote an independent copy of $\xi$, and let $\bar{\xi} = \xi - \xi'$; then $\bar{\xi}$ is 
a symmetric random variable. By symmetry, we have
$$
|\phi(t)|^2 = \E \exp (2 \pi i t \bar{\xi}) = \E \cos (2 \pi t \bar{\xi}). 
$$
Using the inequality $|x| \le \exp \big[ -\frac{1}{2} (1-x^2) \big]$ which is valid for all $x \in \R$, we obtain
\begin{equation}							\label{phi t}
|\phi(t)| \le \exp \Big[ -\frac{1}{2} \big(1 - \E \cos (2 \pi t \bar{\xi}) \big) \Big].
\end{equation}

By assumption, we have $\LL(\xi,1) \le 1-p_0$. Conditioning on $\bar{\xi}$ we see that
$\P \{ |\bar{\xi}| \ge 1 \} \ge p_0$. Furthermore, another assumption of the theorem implies that 
$\E |\bar{\xi}| \le 2 \E |\xi| \le 2M_1$. Using Markov's inequality, we conclude that 
$\P \{ |\bar{\xi}| \ge 4M_1/p_0 \} \le p_0/2$. Combining the two probability bounds, we see that the event
$$
\EE := \big\{ 1 \le |\bar{\xi}| \le C_0 \big\} 
\quad \text{satisfies} \quad
\P \{ \EE \} \ge p_0/2, 
\quad \text{where} \quad
C_0 := \frac{4M_1}{p_0}.
$$
We then estimate the expectation appearing in \eqref{phi t} by conditioning on $\EE$:
\begin{align*}
1 - \E \cos (2 \pi t \bar{\xi})
  &\ge \P\{\EE\} \cdot \E \big[ 1 - \cos (2 \pi t \bar{\xi}) \,\big|\, \EE \big] \\
  &\ge \frac{p_0}{2} \cdot \E \Big[ \frac{4}{\pi^2} \min_{q \in \Z} | 2 \pi t \bar{\xi} - 2 \pi q |^2 \,\big|\, \EE \Big] \\
  &= 8p_0 \, \E \Big[ \min_{q \in \Z} | t \bar{\xi} - q |^2 \,\big|\, \EE \Big].
\end{align*}
Substituting this into \eqref{phi t} and then into \eqref{LSe product}, and using Jensen's inequality, we obtain
\begin{align*}
\LL(S,\e) 
&\le \Const{Esseen} \int_{-1}^1 \exp \Big( -4p_0 
  \E \Big[ \min_{q_k \in \Z} \sum_{k=1}^n \Big| \frac{\bar{\xi} \theta}{\e} x_k - q_k \Big|^2 \,\Big|\, \EE \Big] 
  \Big) d\theta \\
&\le \Const{Esseen} \, \E \Big[ \int_{-1}^1 \exp \Big( -4p_0 
  \dist \Big( \frac{\bar{\xi} \theta}{\e} x, \Z^n \Big)^2 \Big) d\theta \,\Big|\, \EE \Big].
\end{align*}
Since the integrand is an even function of $\theta$, we can integrate over $[0,1]$ instead of $[-1,1]$ at the cost of an extra 
factor of $2$. Also, replacing the expectation by the maximum and using the definition of the 
event $\EE$, we obtain 
\begin{equation}							\label{LLe via f}
\LL(S,\e) \le 2 \Const{Esseen} \sup_{1 \le z \le C_0} \int_0^1 \exp \big( -4p_0 f_z^2(\theta) \big) \,d\theta
\end{equation}
where
$$
f_z(\theta) = \dist \Big( \frac{z\theta}{\e} x, \Z^n \Big).
$$

Suppose that 
$$
\e > \e_0 := \frac{C_0}{D_L(x)}.
$$
Then, for every $1 \le z \le C_0$ and every $\theta \in [0,1]$, we have
$\frac{z\theta}{\e} < D_L(x)$. By the definition of $D_L(x)$, this means that
$$
f_z(\theta) = \dist \Big( \frac{z\theta}{\e} x, \Z^n \Big) \ge L \sqrt{ \log_+ \Big( \frac{z\theta}{\e L} \Big) }.
$$
Putting this estimate back into \eqref{LLe via f}, we obtain
$$
\LL(S,\e) \le 2\Const{Esseen} \sup_{z \ge 1} \int_0^1 \exp \Big( -4p_0 L^2 \log_+ \Big( \frac{z\theta}{\e L} \Big) \Big) \,d\theta.
$$
After change of variable $t = \frac{z\theta}{\e L}$ and using that $z \ge 1$ we have
$$
\LL(S,\e) \le 2\Const{Esseen} L\e \int_0^\infty \exp \big( -4p_0 L^2 \log_+ t \big) \,dt
= 2\Const{Esseen} L\e \Big( 1 + \int_1^\infty t^{-4p_0 L^2} \, dt \Big).
$$
Since $p_0L^2 \ge 1$ by assumption, the integral in the right hand side is bounded by an absolute constant, 
so
$$
\LL(S,\e) \le C_1 L\e
$$
where $C_1$ is an absolute constant.

Finally, suppose that $\e \le \e_0$. Applying the previous part for $2\e_0$, we get
$$
\LL(S,\e) \le \LL(S,2\e_0) \le 2C_1 L\e_0
= \frac{2C_1 C_0 L}{D_L(x)}.
$$
This completes the proof of Theorem~\ref{sbp via lcd}.
\end{proof}

\begin{remark}				\label{sbp via lcd unnormed}
  For a general, not necessarily unit vector $x \in \R^n$, the conclusion of Theorem~\ref{sbp via lcd}
  reads as
  $$
  \LL(S,\e) = \LL \Big( \frac{S}{\|x\|_2}, \, \frac{\e}{\|x\|_2} \Big)   
  \le \Const{sbp via lcd} L \Big( \frac{\e}{\|x\|_2} + \frac{1}{D_L(x/\|x\|_2)} \Big).
  $$  
\end{remark}

\subsection{Regularized LCD}				\label{s: reg LCD}

As we saw in Proposition~\ref{dist via quadratic}, the distance problem reduces 
to a quadratic Littlewood-Offord problem, for quadratic forms of the type $\sum_{ij} x_{ij} \xi_i \xi_j$.
We will seek to reduce the quadratic problem to a linear one by decoupling and conditioning arguments. 
This process requires a more robust version of the concept of the LCD, which we develop now. 

Let $x \in \Incomp(c_0,c_1)$; recall that we have fixed the values $c_0 = c_0(K, M_4)$, $c_1 = c_1(K, M_4)$
in Remark~\ref{c0 c1}. 
By Lemma~\ref{incomp spread}, at least $\frac{1}{2} c_0 c_1^2 n$ coordinates $x_k$ of $x$
satisfy
\begin{equation}							\label{eq spread}
\frac{c_1}{\sqrt{2n}} \le |x_k| \le \frac{1}{\sqrt{c_0 n}}.
\end{equation}
Let us fix some constant $c_{oo}$ such that 
$$
\frac{1}{4} c_0 c_1^2 \le c_{oo} \le \frac{1}{4};
$$
we can make the value of $c_{oo}$ depend only on $c_0$ and $c_1$ (hence only on parameters $K$ and $M_4$).
Then for every vector $x \in \Incomp(c_0,c_1)$ we can assign a subset called $\spread(x) \subseteq[n]$ 
so that 
$$
|\spread(x)| = \lceil c_{oo} n \rceil
$$
and so that \eqref{eq spread} holds for all $k \in \spread(x)$. 

The point here is that not all of the coordinates $x_k$ satisfying \eqref{eq spread} will be good in the future; 
the set $\spread(x)$ will allow us to include only the good ones. At this point, we consider an arbitrary valid 
assignment of $\spread(x)$ to $x$; the particular choice of the assignment will be determined later. 

Our new version of LCD is designed to capture the amount of structure in the 
{\em least structured} part of the coefficients of $x$. 

\begin{definition}[Regularized LCD]					\label{def reg LCD}
  Let $\l \in (0, c_{oo})$ and $L \ge 1$. 
  We define the {\em regularized LCD} of a vector $x \in \Incomp(c_0,c_1)$ as
  $$
  \Dhat_L(x,\l) = \max \Big\{ D_L \big(x_I/\|x_I\|_2\big) : \, I \subseteq \spread(x), \, |I| = \lceil \l n \rceil \Big\}.
  $$
  Denote by $I(x)$ the maximizing set $I$ in this definition.
\end{definition}

\begin{remark}				\label{norm xI}
  Since the sets $I$ in this definition are subsets of $\spread(x)$, inequalities \eqref{eq spread}
  imply that 
  $$
  \const{norm xI} \sqrt{\l} \le \|x_I\|_2 \le \Const{norm xI} \sqrt{\l}
  $$
  where $\const{norm xI} = c_1/\sqrt{2}$ and $\Const{norm xI} = 1/\sqrt{c_0}$. 
\end{remark}

\begin{lemma}					\label{LCD sqrtn}
  For every $x \in \Incomp(c_0,c_1)$ and every $\l \in (0, c_{oo})$ and $L \ge 1$, one has
  $$
  \Dhat_L(x,\l) \ge \const{LCD sqrtn} \sqrt{\l n}.
  $$
  Here $\const{LCD sqrtn} \in (0,1)$ depends only on $c_0$ and $c_1$. 
\end{lemma}

\begin{proof}
Consider a subset $I$ as in the definition of $\Dhat_L(x,\l)$.
Denote $z_I := x_I/\|x_I\|_2$. By \eqref{eq spread} and Remark~\ref{norm xI}, we have
$\|z_I\|_\infty \le C/\sqrt{\l n}$ where $C \in (0,1)$ depends only on $c_0$ and $c_1$. 
Then Lemma~\ref{usual LCD large} implies that 
$$
D_L(z_I) \ge \frac{1}{2C} \sqrt{\l n}. 
$$
By the definition of $\Dhat_L(x,\l)$, the proof is complete. 
\end{proof}

Now we state a version of Theorem~\ref{sbp via lcd} for the regularized LCD.

\begin{proposition}[Small ball probabilities via regularized LCD]				\label{sbp via reg lcd}
  Let $\xi_1,\ldots,\xi_n$ be independent and identically distributed random variables. Assume that 
  there exist numbers $\e_0, p_0 > 0$ such that $\LL(\xi_k, \e_0) \le 1-p_0$ and $\E |\xi_k| \le M_1$ for all $k$. 
  Then there exist $\Const{sbp via reg lcd}$ which depends only on $\e_0$, $p_0$, and $M_1$, and such that the following holds. 

  Consider a vector $x \in \Incomp(c_0,c_1)$ and a subset $J \subseteq [n]$ such that $J \supseteq I(x)$. 
  Consider also the sum $S_J = \sum_{k \in J} x_k \xi_k$.
  Then for every $\l \in (0,c_{oo})$, $L \ge p_0^{-1/2}$ and $\e \ge 0$, one has
  $$
  \LL(S_J,\e) \le \Const{sbp via reg lcd} L \Big( \frac{\e}{\sqrt{\l}} + \frac{1}{\Dhat_L(x,\l)} \Big).
  $$
\end{proposition}

\begin{proof}
Note that for every two sets $I \subseteq J \subseteq [n]$, the corresponding sums satisfy
$\LL(S_J, \e) \le \LL(S_I, \e)$; this follows by conditioning on the random variables $\xi_k$ with $k \in J \setminus I$.
Applying this relation for $I:= I(x) \subseteq J$, we obtain 
\begin{align*}
\LL(S_J, \e) 
  \le \LL(S_I, \e) 
  &\le \Const{sbp via lcd} L \Big( \frac{\e}{\|x_I\|_2} + \frac{1}{D_L(x_I/\|x_I\|_2)} \Big)  \qquad \text{(by Remark~\ref{sbp via lcd unnormed})} \\
  &\le \Const{sbp via lcd} L \Big( \frac{\e}{\const{norm xI}\sqrt{\l}} + \frac{1}{\Dhat_L(x,\l)} \Big) \qquad \text{(by Remark~\ref{norm xI}).}
\end{align*}
This completes the proof. 
\end{proof}

\begin{remark}				\label{sbp assumptions}
  By Lemma~\ref{Levy weak}, both Theorem~\ref{sbp via lcd} and Proposition~\ref{sbp via reg lcd} 
  can be applied for arbitrary independent and identically
  distributed random variables $\xi_1,\ldots,\xi_n$ that have unit variance and finite fourth moment.
  In particular, Theorem~\ref{sbp via lcd} and Proposition~\ref{sbp via reg lcd} apply 
  if $\xi_k$ satisfy the same moment assumptions \eqref{aij} as the entries 
  $a_{ij}$ of $A$. The constants $\Const{sbp via lcd}$ and $\Const{sbp via reg lcd}$ 
  in this case depends only on the fourth moment parameter $M_4$
  from the assumptions \A on the random matrix $A$.
\end{remark}

\subsection{Small ball probabilities for $Ax$ via regularized LCD}

We will now develop a refinement of Proposition~\ref{sbp Ax} that is sensitive
to the additive structure of the vector $x$.

\begin{proposition}[Small ball probabilities for $Ax$ via regularized LCD]		\label{sbp Ax via LCD}
  Let $A$ be a random matrix which satisfies \A.
  Let $x \in \Incomp(c_0,c_1)$ and $\l \in (0,c_{oo})$. Then for every $L \ge L_0$ and $\e \ge 0$, one has
  $$
  \LL(Ax, \e\sqrt{n}) \le \left[ \frac{\Const{sbp Ax via LCD} L \e}{\sqrt{\l}} 
    + \frac{\Const{sbp Ax via LCD} L}{\Dhat_L(x,\l)} \right]^{n - \lceil \l n \rceil}.
  $$
  Here $\Const{sbp Ax via LCD}$ and $L_0$ depend only on the parameters $K$ and $M_4$ from assumptions 
  \eqref{norm}, \eqref{aij}.
\end{proposition}

\begin{proof}
Our goal is to bound above the probability 
$$
\P \big\{ \|Ax-u\|_2 \le \e \sqrt{n} \big\}
$$
for an arbitrary fixed vector $u \in \R^n$.

Let $I = I(x)$ be the maximizing set from the definition of $\Dhat_L(x,\l)$. 
We decompose the set of indices $[n]$ into sets $I \cup I^c$ similarly to how we did it in the
proof of Proposition~\ref{sbp Ax}. This induces the decomposition 
of the matrix $A$ and both vectors in question,
which we denote  
$$
A = 
\begin{pmatrix} 
  D & G \\ 
  G^* & E 
\end{pmatrix}, \quad
x = \begin{pmatrix} y \\ z \end{pmatrix}, \quad
u = \begin{pmatrix} v \\ w \end{pmatrix},
$$
where $D$ is a $I^c \times I^c$ matrix, $G$ is a $I^c \times I$ matrix, 
$y, v \in \R^{I^c}$ and $z,w \in \R^I$. 
This way, we express
$$
\|Ax-u\|_2^2 = \|Dy + Gz - v\|_2^2 + \|G^*y + Ez - w\|_2^2.
$$
Let us condition on an arbitrary realization of the minors $D$ and $E$. 
Denoting $u_0 := v - Dy$, we have 
$$
\|Ax-u\|_2 \ge \|Gz-u_0\|_2.
$$

We will use the crucial facts that $G$ is a $I^c \times I$ matrix with independent entries, 
and $u_0$ is a fixed vector in $\R^{I^c}$. The $i$-th coordinate of the vector $Gz \in \R^{I^c}$ is 
$$
(Gz)_i = \sum_{j \in I} a_{ij} x_j, \quad i \in I^c.
$$
All random variables $a_{ij}$ here are independent. So we can apply Proposition~\ref{sbp via reg lcd} 
with $J = I = I(x)$ (see Remark~\ref{sbp assumptions}), and we obtain
$$
\LL \big( (Gz)_i, \e \big) \le \Const{sbp via reg lcd} L \Big( \frac{\e}{\sqrt{\l}} + \frac{1}{\Dhat_L(x,\l)} \Big),  \quad i \in I^c.
$$
Since the coordinates $(Gz)_i$ of the random vector $Gz$ are independent,
Tensorization Lemma~\ref{tensorization} (see Remark~\ref{tensorization remark}) implies that 
$$
\LL \big( Gz, \e \sqrt{|I^c|} \big) \le \left[ \frac{CL\e}{\sqrt{\l}} + \frac{CL}{\Dhat_L(x,\l)} \right]^{|I^c|},
$$
where $C$ depends on $\Const{sbp via reg lcd}$ only.
This concludes the proof since $|I^c| = n - \lceil \l n \rceil \ge n/2$.
\end{proof}

\section{Estimating additive structure}

Recall that our goal is to estimate the small ball probabilities for the quadratic forms of the type
$\< A^{-1} X, X\> $. In accordance with the spirit of Littlewood-Offord theory, 
we will first need to estimate the amount of additive structure in the random vector $A^{-1} X$. 
In this section, we indeed show that the regularized LCD of $A^{-1}X$ is large for every fixed $X$. 
This will be used later along with a decoupling argument to bound the small ball probabilities for
$\< A^{-1} X, X\> $.

Recall that the values of constants $c_0, c_1, c_{oo}$ are already chosen in Remark~\ref{c0 c1}; they 
depend only on parameters $K$, $M_4$.

\begin{theorem}[Structure theorem]					\label{structure}
  Let $A$ be a random matrix which satisfies \A.
  There exist $\const{structure}>0$ and $L_0 \ge 1$ that depend only on the parameters $K$ and $M_4$ 
  from assumptions \eqref{norm}, \eqref{aij}, and such that the following holds. 
  Let $u \in \R^n$ be an arbitrary fixed vector, and consider 
  $x_0 := A^{-1}u / \|A^{-1}u\|_2$. 
  Let $L \ge L_0$ and $n^{-\const{structure}} \le \l \le c_{oo}/3$.
  Consider the event 
  $$
  \EE = \Big\{ x_0 \in \Incomp(c_0,c_1) \text{ and } \Dhat_L(x_0, \l) \ge L^{-2} n^{\const{structure}/\l} \Big\}.
  $$
  Then
  $$
  \P (\EE^c \cap \EE_K) \le 2e^{-\const{structure}n}. 
  $$
\end{theorem}

We shall first prove the easier fact that $x_0 \in \Incomp(c_0,c_1)$. The more difficult part of the theorem
is the estimate on the LCD. Its proof will be based on the probability bound of Proposition~\ref{sbp Ax via LCD}
and nontrivial covering estimates for the sets of vectors with given LCD, which we shall develop in Section~\ref{s: covering}. 

\begin{lemma}[$A^{-1} u$ is incompressible]						\label{inverse incompressible}
  In the setting of Theorem~\ref{structure}, consider the event
  $$
  \EE_1 = \big\{ x_0 \in \Incomp(c_0,c_1) \big\}.
  $$
  Then
  $$
  \P(\EE_1^c \cap \EE_K) \le 2e^{-\const{inverse incompressible} n}.
  $$
  Here $\const{inverse incompressible} > 0$ depends only on the parameters $K$ and $M_4$ from assumptions \eqref{norm}, \eqref{aij}.
\end{lemma}

\begin{proof}
Denote $x = A^{-1}u$; then $Ax=u$. Therefore
$$
\EE_1^c \subseteq \Big\{ \exists x \in \R^{n}: \, \frac{x}{\|x\|_2} \in \Comp(c_0,c_1) \wedge Ax=u \Big\}.
$$
By Proposition~\ref{sbp comp}, $\P(\EE_1^c \cap \EE_K) \le 2e^{-\const{sbp comp} n}$
as claimed. 
\end{proof}

\subsection{Covering sets of vectors with small LCD}				\label{s: covering}

\begin{definition}[Sublevel sets of LCD]					\label{SD}	
  Let us fix $\l \in (0, c_{oo})$. For every value $D \ge 1$, we define the set
  $$
  S_D = \big\{ x \in \Incomp(c_0,c_1) :\, \Dhat_L(x,\l) \le D \big\}.
  $$
\end{definition}

Our present goal is to bound the covering numbers of $S_D$. 

\begin{proposition}[Covering sublevel sets of regularized LCD]		\label{reg LCD net}
  There exist $\Const{reg LCD net}, \const{reg LCD net} > 0$ 
  which depend only on $c_0,c_1$, and such that the following holds. 
  Let $\l \in (\Const{reg LCD net}/n, c_{oo}/3)$ and $L \ge 1$. 
  For every $D \ge 1$, the sublevel set $S_D$ has a $\b$-net $\NN$ such that 
  $$
  \b = \frac{L \sqrt{\log D}}{\sqrt{\l} D}, \quad
  |\NN| \le \left[ \frac{\Const{reg LCD net} D}{(\l n)^{\const{reg LCD net}}} \right]^n D^{1/\l}.
  $$  
\end{proposition}

The main point of this result is the presence of the term $(\l n)^{\const{reg LCD net}} \gg 1$ in the estimate
of the cardinality of $\NN$. This makes $|\NN|$ substantially smaller than $(3/\b)^n$, 
which is a trivial estimate on the $\b$-net for the whole sphere $S^{n-1}$, see Lemma~\ref{covering sphere}.

The proof of Proposition~\ref{reg LCD net} relies on a series of lemmas of increasing generality. 
We begin by covering the level sets of the usual (not regularized) LCD. 
We shall work in a lower dimension $m$ for the time being; the definition of LCD is thus considered in $\R^m$. 

\begin{lemma}				\label{LCD net two-sided}
  Let $c \in (0,1)$, $D_0 \ge c \sqrt{m} \ge 1$ and $L \ge 1$. 
  Then the set 
  $$
  \big\{ x \in S^{m-1}:\, D_L(x) \in (D_0, 2D_0] \big\} 
  $$
  has a $\beta$-net $\NN$ such that 
  $$
  \b = \frac{2 L \sqrt{\log(2D_0)}}{D_0}, \quad
  |\NN| \le \left( \frac{CD_0}{\sqrt{m}} \right)^m.
  $$
  Here $C$ depends only on $c$. 
\end{lemma}

\begin{proof}
Let $x$ be a vector from the set in question. By the definition of LCD, there exists $p \in \Z^m$ such that 
\begin{equation}							\label{dist Dxx}
\|D_L(x)x-p\|_2 \le L \sqrt{\log_+(2D_0/L)}.
\end{equation}
Dividing both sides by $D_L(x)$ and using trivial estimates in the right hand side, we get
$$
\left\| x - \frac{p}{D_L(x)} \right\|_2 \le \frac{L \sqrt{\log(2D_0)}}{D_0}.
$$
Since $\|x\|_2 = 1$, the last two inequalities imply that 
$$
\left\| x - \frac{p}{\|p\|_2} \right\|_2 \le \frac{2L \sqrt{\log(2D_0)}}{D_0}.
$$
Moreover, since $\|x\|_2=1$, we have
$$
\|p\|_2 
\le \|D_L(x)x-p\|_2 + \|D_L(x)x\|_2
\le L \sqrt{\log_+(2D_0/L)} + 2D_0
\le 4D_0.
$$
This shows that the set
$$
\NN := \Big\{ \frac{p}{\|p\|_2} :\, p \in \Z^n \cap 4 D_0 B_2^m \Big\}
$$
is indeed an $\b$-net of the set in question. 
Counting the number of integer points in a ball by a standard volume argument, we
estimate 
$$
|\NN| \le \left( 1 + \frac{12 D_0}{\sqrt{m}} \right)^m
\le \left( \frac{CD_0}{\sqrt{m}} \right)^m.
$$
This completes the proof.
\end{proof}

The next step is toward removing the lower bound for $D_L(x)$ in Lemma~\ref{LCD net two-sided}. 

\begin{lemma}				\label{LCD net two-sided D}
  Let $c \in (0,1)$, $D \ge D_0 \ge c \sqrt{m} \ge 1$ and $L \ge 1$. 
  Then the set 
  $$
  \big\{ x \in S^{m-1}:\, D_L(x) \in (D_0, 2D_0] \big\} 
  $$
  has a $\beta$-net $\NN$ such that 
  $$
  \b = \frac{4 L \sqrt{\log(2D)}}{D}, \quad
  |\NN| \le \left( \frac{CD}{\sqrt{m}} \right)^m.
  $$
  Here $C$ depends only on $c$. 
\end{lemma}

\begin{proof}
By Lemma~\ref{LCD net two-sided}, we can cover the set in question with $\big( \frac{C_0 D_0}{\sqrt{m}} \big)^m$
Euclidean balls of radius $\b_0 = \frac{2 L \sqrt{\log(2D_0)}}{D_0}$ centered in the set, where $C_0$ depends only on $c$.
If $\b_0 \le \b$ then the lemma is proved.
Assume that $\b_0 \ge \b$. We can further cover every ball of radius $\b_0$ by balls of the smaller radius $\b/2$. 
According to Lemma~\ref{covering sphere}, the number of smaller balls per larger ball is at most 
$$
\left( 1 + \frac{4\b_0}{\b} \right)^m \le \left( \frac{5\b_0}{\b} \right)^m \le \left( \frac{3D}{D_0} \right)^m.
$$
The total number of smaller balls is then at most 
$$
\left( \frac{C_0 D_0}{\sqrt{m}} \right)^m \cdot \left( \frac{3D}{D_0} \right)^m
\le \left( \frac{3C_0 D}{\sqrt{m}} \right)^m.
$$
By enlarging the radius of the balls from $\b/2$ to $\b$ as in Remark~\ref{centering}, 
one can assume that they are centered in the set in question. 
This completes the proof. 
\end{proof}

Now we can remove the flexible lower bound on $D_L(x)$ in Lemma~\ref{LCD net two-sided}.

\begin{lemma}				\label{LCD net}
  Let $c \in (0,1)$ such that $D > c \sqrt{m} \ge 2$ and $L \ge 1$. 
  Then the set 
  $$
  \big\{ x \in S^{m-1}:\, c\sqrt{m} < D_L(x) \le D \big\} 
  $$
  has a $\beta$-net $\NN$ such that 
  $$
  \b = \frac{4 L \sqrt{\log(2D)}}{D}, \quad
  |\NN| \le \left( \frac{CD}{\sqrt{m}} \right)^m \log_2 D.
  $$
  Here $C$ depends only on $c$. 
\end{lemma}

\begin{proof}
We decompose the set
$$
\big\{ x \in S^{m-1}:\, D_L(x) \le D \big\} 
\subseteq \bigcup_k \big\{ x \in S^{m-1}:\, D_L(x) \in (2^{-k}D, 2^{-k+1}D] \big\},
$$
where the union is over the integers $k$ such that the interval $(2^{-k}D, 2^{-k+1}D]$
has a nonempty intersection with the interval $(c\sqrt{m}, D]$.
The assumptions imply that all such $k$ are nonnegative and $2^{-k}D \ge c\sqrt{m}/2 \ge 1$. 
So there are at most $\log_2 D$ terms in this union, and for each term one can construct an $\b$-net 
using Lemma~\ref{LCD net two-sided D}. 
The union of these nets forms a required net $\NN$.
\end{proof}

Further, we remove the normalization requirement from the set to be covered. 

\begin{lemma}				\label{D net}
  Let $c \in (0,1)$ such that $D > c \sqrt{m} \ge 2$ and $L \ge 1$. 
  Then the set 
  \begin{equation}							\label{set covered}
  \big\{ x \in B_2^m:\, c\sqrt{m} < D_L(x/\|x\|_2) \le D \big\} 
  \end{equation}
  has a $\beta$-net $\NN$ such that 
  $$
  \b = \frac{4 L \sqrt{\log(2D)}}{D}, \quad
  |\NN| \le \left( \frac{CD}{\sqrt{m}} \right)^m D^2.
  $$
  Here $C$ depends only on $c$. 
\end{lemma}

\begin{proof}
Let $\NN_0$ be a $\b$-net of the set $\big\{ x \in S^{m-1}:\, c\sqrt{m} < D_L(x) \le D \big\}$
as in Lemma~\ref{LCD net}. For each $x \in \NN_0$, let $\MM_x$ denote a $\b/2$-net 
of the interval $\Span(x) \cap B_2^m$ such that $|\MM_x| \le 4/\b$. 
Then $\NN := \cup_{x \in \NN_0} \MM_x$ clearly forms a $\b$-net of the set in \eqref{set covered}, 
and 
$$
|\NN| \le |\NN_0| \cdot \frac{4}{\b} 
\le \left( \frac{CD}{\sqrt{m}} \right)^m \log_2 D \cdot \frac{D}{L \sqrt{\log(2D)}}.
$$
A trivial estimate of the right hand side completes the proof. 
\end{proof}

\medskip

\begin{proof}[Proof of Proposition~\ref{reg LCD net}.] \quad

\noindent {\em Step 1: decomposition.}
Consider a vector $x \in S_D$. Recall from Section~\ref{s: reg LCD} that 
$\spread(x) \subseteq [n]$ and $|\spread(x)| = \lceil c_{oo} n \rceil$.
Let us decompose $\spread(x)$ into disjoint sets
$$
\spread(x) = I_1 \cup \cdots \cup I_{k_0} \cup J
$$
for some $k_0$ such that 
$$
|I_k| = \lceil \l n \rceil \text{ for } k \le k_0, \quad |J| < \lceil \l n \rceil,
$$
and so that the sets fill $\spread(x)$ from left to right, i.e. 
$\sup I_k < \inf I_{k+1}$ and $\sup I_k < \inf J$ for all $k$. 
Since $\l \le c_{oo}$, we have $k_0 \ge 1$. 
Moreover, let 
$$
I_0 = [n] \setminus (I_1 \cup \cdots \cup I_{k_0}). 
$$
This produces a decomposition of $[n]$ into disjoint sets 
\begin{equation}							\label{n decomposition}
[n] = I_0 \cup I_1 \cup \cdots I_{k_0}.
\end{equation}
This decomposition is obviously uniquely determined by the subset $\spread(x)$, and it does
not otherwise depend on $x$. 

We notice two useful bounds that will help us later. 
Since $I_1 \cup \cdots \cup I_{k_0} = \spread(x) \setminus J$, we have
\begin{equation}							\label{union Ik}
|I_1 \cup \cdots \cup I_{k_0}| 
\ge \lceil c_{oo} n \rceil - \lceil \l n \rceil
\ge c_{oo} n / 2
\end{equation}
and
\begin{equation}							\label{k0 small}
k_0 \le \frac{\lceil c_{oo} n \rceil}{\lceil \l n \rceil} \le \frac{2c_{oo}}{\l}.
\end{equation}

{\em Step 2: constructing nets for each component.}
Let consider a fixed decomposition \eqref{n decomposition}, and decompose the vector $x$ 
accordingly:
$$
x = (x_{I_0}, x_{I_1}, \ldots, x_{I_{k_0}}).
$$
We are going to construct separate $\b$-nets for each component $x_{I_k}$, 
and combine them in to one net for $S_D$.

A net $\NN_0$ for the first component of $x$ is chosen trivially. Note that $x_{I_0} \in B_2^{I_0}$.
By Lemma~\ref{covering sphere}, we can choose a $(1/D)$-net $\NN_0$ of $B_2^{I_0}$ with 
$$
|\NN_0| \le (3D)^{|I_0|}. 
$$

For the other components of $x$, we will choose $\b_0$-nets non-trivially, where
\begin{equation}							\label{beta0}
\b_0 = \frac{4 L \sqrt{\log(2D)}}{D}.
\end{equation}
To this end, let us fix $k \le k_0$. Since $x \in S_D$, the definition of the regularized LCD yields that 
$$
D_L \big( x_{I_k}/\|x_{I_k}\|_2 \big) \le \Dhat_L(x,\l) \le D.
$$
On the other hand, the argument in Lemma~\ref{LCD sqrtn} yields
$$
D_L \big( x_{I_k}/\|x_{I_k}\|_2 \big) \ge \const{LCD sqrtn} \sqrt{\l n}.
$$
By the assumptions, 
$$
\const{LCD sqrtn} \sqrt{\l n} \ge \frac{\const{LCD sqrtn}}{2} \sqrt{|I_k|} \ge \frac{\const{LCD sqrtn}}{2} \sqrt{\l n} \ge 2.
$$
(We can choose a value of $\Const{reg LCD net}$ large enough so that this holds). 
Thus
$$
D_L \big( x_{I_k}/\|x_{I_k}\|_2 \big) \ge \frac{\const{LCD sqrtn}}{2} \sqrt{|I_k|} \ge 2.
$$
We have shown that $x_{I_k}$ belongs to the set
$$
V_k := \big\{ y \in B_2^{I_k}:\, \frac{\const{LCD sqrtn}}{2} \sqrt{|I_k|} < D_L(y/\|y\|_2) \le D \big\}.
$$ 
By Lemma~\ref{D net}, there exists a $\b_0$-net $\NN_k$ of $V_k$ with 
$$
\Big( \frac{CD}{\sqrt{|I_k|}} \Big)^{|I_k|} D^2.
$$

{\em Step 3: combining the nets.}
We are going to combine the nets $\NN_k$ into one net for $S_D$.
So far we have shown that for every $x \in S_D$, there exist a decomposition \eqref{n decomposition}
and nets $\NN_0, \NN_1, \ldots, \NN_{k_0}$ which are uniquely determined by the index set $\spread(x)$, 
and there exist vectors $y_k \in \NN_k$ such that 
$$
\|x_{I_k} - y_{I_k}\|_2 \le \b_0, \quad k = 0,1,\ldots, k_0.
$$
Consider the vector
\begin{equation}							\label{y decomposed}
y = (y_{I_0}, y_{I_1}, \ldots, y_{I_{k_0}}).
\end{equation}
It follows that
$$
\|x-y\|_2 = \Big( \sum_{k=0}^{k_0} \|x_{I_k} - y_{I_k}\|_2^2 \Big)^{1/2}
\le \b_0 \sqrt{k_0+1}.
$$
By \eqref{k0 small} and since $\l \le c_{oo}$, we have $k_0+1 \le 3c_{oo}/\l$.
Recalling the definition \eqref{beta0} of $\b_0$, we conclude that
$$
\|x-y\|_2 \le \frac{7 \sqrt{c_{oo}} L \sqrt{\log L}}{\sqrt{\l} D} 
\le \frac{L \sqrt{\log(2D)}}{\sqrt{\l} D} = \b,
$$
where we used that the value of $c_{oo}$ can be chosen small enough (smaller than $1/49$ in this case). 

Consider the set $\NN$ of vectors $y$ that can arise in \eqref{y decomposed}.
We showed that $\NN$ is a $\b$-net of $S_D$. 
Moreover, since the index set $\spread(x)$ can be chosen in at most $2^n$ ways, it follows that 
$$
|\NN| \le 2^n |\NN_0| |\NN_1| \cdots |\NN_{k_0}|
\le 2^n \cdot (3D)^{|I_0|} \cdot \prod_{k=1}^{k_0} \Big( \frac{CD}{\sqrt{|I_k|}} \Big)^{|I_k|} D^2.
$$ 
To simplify this bound, note that $\sum_{k=1}^{k_0} |I_k| \ge c_{oo} n/ 2$ by \eqref{union Ik}
and that $\sum_{k=0}^{k_0} |I_k| = n$ and $|I_k| \ge \l n \ge 1$ by construction. It follows that 
$$
|\NN| \le \frac{(6CD)^n}{(\sqrt{\l n})^{c_{oo}n/2}} D^{2k_0}.
$$
Estimate \eqref{k0 small} on $k_0$ implies that $2k_0 \le 1/\l$, which 
completes the proof of Proposition~\ref{reg LCD net}.
\end{proof}

\subsection{Proof of Structure Theorem~\ref{structure}.}

In Proposition~\ref{sbp Ax via LCD} we estimated the small ball probabilities
for the random vector $Ax$ for a fixed vector $x$. 
Now we combine this with the covering results of the previous section to obtain 
a bound that is uniform over all $x$ with small regularized LCD. 
Recall that $S_D$ denotes the sub-lebel set of regularized LCD according to Definition~\ref{SD}.
	
\begin{lemma}[Small ball probabilities on a sublevel set of LCD]				\label{sbp Ax sublevel}
  There exist $c, c'>0$ and $L_0 \ge 1$ that depend only on the parameters $K$ and $M_4$ 
  from the assumptions \eqref{norm}, \eqref{aij}, and such that the following holds. 
  Let $L \ge L_0$, $n^{-c} \le \l \le c_{oo}/3$ and $1 \le D \le L^{-2} n^{c/\l}$. 
  Then 
  \begin{equation}							\label{eq sbp Ax sublevel}
  \P \big\{ \exists x \in S_D:\, \|Ax-u\|_2 \le K \b \sqrt{n} \wedge \EE_K \big\} \le n^{-c'n},
  \end{equation}
  where 
  $$
  \b = \frac{L \sqrt{\log(2D)}}{\sqrt{\l} D}.
  $$
\end{lemma}

\begin{proof}
We will first compute the probability for $S_{D} \setminus S_{D/2}$ instead of $S_D$ in \eqref{eq sbp Ax sublevel}. 
Proposition~\ref{sbp Ax via LCD} implies that for every $x \in S_{D} \setminus S_{D/2}$, one has
$$
\P \big\{ \|Ax-u\|_2 \le \e \sqrt{n} \big\} 
\le \left[ \frac{\Const{sbp Ax via LCD} L \e}{\sqrt{\l}} + \frac{\Const{sbp Ax via LCD} L}{D} \right]^{n - \lceil \l n\rceil}, \quad \e \ge 0.
$$
Let us use this inequality for $\e = 4 K \b$.
Clearly, the term $\frac{\e}{\sqrt{\l}}$ dominates the term $\frac{1}{4D}$.
So we obtain
$$
\P \big\{ \|Ax-u\|_2 \le 4 K \b \sqrt{n} \big\} 
\le \left[ \frac{C' L^2 \sqrt{\log(2D)}}{\l D} \right]^{n - \lceil \l n\rceil} =: p_0.
$$
(Here the constant $C' = C'(K, M_4)$ absorbs the factor $K$.)

Let us choose a $\b$-net $\NN$ of $S_{D} \setminus S_{D/2}$ according to Proposition~\ref{reg LCD net}. 
The union bound yields
\begin{align*}
\P \big\{ \exists x \in \NN &:\, \|Ax-u\|_2 \le 4 K \b \sqrt{n} \big\} 
  \le |\NN| \cdot p_0\\
  &\le \left[ \frac{\Const{reg LCD net} D}{(\l n)^{\const{reg LCD net}}} \right]^n D^{1/\l} 
    \cdot \left[ \frac{C' L^2 \sqrt{\log(2D)}}{\l D} \right]^{n - \lceil \l n\rceil} =: p_1.
\end{align*}
One can estimate $p_1$ using the assumptions that $n$ is sufficiently large, 
$n^{-c} \le \l \le c_{oo}/3$ and $1 \le D \le L^{-2} n^{c/\l}$.
Choosing the constant $c>0$ sufficiently small and making simplifications, we obtain
$$
p_1 \le n^{-c'' n}.
$$
 
Suppose event $\EE_K$ occurs, and suppose there exists $x \in S_{D} \setminus S_{D/2}$ 
such that $\|Ax-u\|_2 \le K \b \sqrt{n}$.
There exists $x_0 \in \NN$ such that $\|x-x_0\|_2 \le \b$. Then
\begin{align*}
\|Ax_0-u\|_2 
  &\le \|Ax-u\|_2 + \|A(x-x_0)\|_2
  \le \|Ax-u\|_2 + \|A\| \|x-x_0\|_2 \\
  &\le K \b \sqrt{n} + 3 K \sqrt{n} \cdot \b
  = 4 K \b \sqrt{n}.
\end{align*}
As we know, the probability of the latter event is at most $p_1 \le n^{-c'' n}$. So we have shown that
$$
\P \big\{ \exists x \in S_{D} \setminus S_{D/2}:\, \|Ax-u\|_2 \le K \b \sqrt{n} \wedge \EE_K \big\} 
\le n^{-c'' n}. 
$$

Finally, we get rid of $S_{D/2}$ in this bound. 
Since $\b$ decreases in $D$, as long as $D/2 \ge 1$ the previous result can be applied for $D/2$ instead of $D$, 
and we get 
$$
\P \big\{ \exists x \in S_{D/2} \setminus S_{D/4}:\, \|Ax-u\|_2 \le K \b \sqrt{n} \wedge \EE_K \big\} 
\le n^{-c'' n}. 
$$
We can continue this way for $S_{D/4} \setminus S_{D/8}$, etc.
So we decompose $S = \bigcup_{k=0}^{k_0} (S_{2^{-k} D} \setminus S_{2^{-k-1} D})$, 
where $k_0$ is the largest integer such that $2^{-k_0} D \ge \const{LCD sqrtn} \sqrt{\l n}$.
(Recall that by Proposition~\ref{LCD sqrtn}, the set $S_{D_0}$ is empty if $D_0 <  \const{LCD sqrtn} \sqrt{\l n}$.
Since $\const{LCD sqrtn} \sqrt{\l n} \ge 1$, we have $k_0 \le \log_2 D$. The union bound then gives
$$
\P \big\{ \exists x \in S_D:\, \|Ax-u\|_2 \le K \b \sqrt{n} \wedge \EE_K \big\}
\le k_0 \cdot n^{-c'' n} 
\le \log_2(D) n^{-c'' n} 
\le n^{-c'n}
$$
if the constant $c'>0$ is chosen appropriately small.
This completes the proof. 
\end{proof}

\begin{proof}[Proof of Structure Theorem~\ref{structure}.]
We fix constants $c, c', L_0$ given by Lemma~\ref{sbp Ax sublevel}. 
Consider the following two events: 
\begin{gather*}
\EE_0 = \big\{ \Dhat_L(x_0, \l) > L^{-2} n^{c/\l} =: D_0 \text{ or } \Dhat_L(x_0, \l) \text{ is undefined} \big\}, \\
\EE_1 = \big\{ x_0 \in \Incomp(c_0,c_1) \big\}.
\end{gather*}
Recall that if $\EE_1$ holds then $\Dhat_L(x_0, \l)$ is defined. So our desired event $\EE$ can be written as 
$$
\EE = \EE_1 \cap \EE_0.
$$
Then $\EE^c = \EE_1^c \cup (\EE_1 \cap \EE^c) = \EE_1^c \cup (\EE_1 \cap \EE_0^c)$. 
Finally, the event whose probability we need to estimate is 
$\EE^c \cap \EE_K \subseteq (\EE_1^c \cap \EE_K) \cup (\EE_1 \cap \EE_0^c \cap \EE_K)$. Hence
$$
\P (\EE^c \cap \EE_K) \le \P (\EE_1^c \cap \EE_K) + \P (\EE_1 \cap \EE_0^c \cap \EE_K).
$$
The first term was estimated in Lemma~\ref{inverse incompressible} as 
$$
\P(\EE_1^c \cap \EE_K) \le 2e^{-\const{inverse incompressible} n}.
$$
It remains to obtain a similar estimate on the second term $\P (\EE_1 \cap \EE_0^c \cap \EE_K)$.
We can express
$$
\EE_1 \cap \EE_0^c \cap \EE_K
= \big\{ x_0 := A^{-1}u / \|A^{-1}u\|_2 \in S_{D_0} \wedge \EE_K \big\}.
$$
Let $u_0 := Ax_0 = u/\|A^{-1}u\|_2$. Event $\EE_K$ implies that
$\|u_0\|_2 = \|Ax_0\|_2 \le \|A\| \le 3 K \sqrt{n}$. Therefore $u_0$ lies on a one-dimensional interval: 
$$
u_0 \in \Span(u) \cap 3 K \sqrt{n} B_2^n =: E.
$$
So
$$
\EE_1 \cap \EE_0^c \cap \EE_K
\subseteq \big\{ \exists x_0 \in S_{D_0}, \, \exists u_0 \in E :\, Ax_0 = u_0 \wedge \EE_K \}.
$$

In view of an application of Lemma~\ref{sbp Ax sublevel}, let us choose
$$
\b_0 = \frac{L \sqrt{\log(2D_0)}}{D_0}.
$$
Let $\MM$ denote some fixed $(K \b_0 \sqrt{n})$-net of the interval $E$, such that 
$$
|\MM| \le \frac{3 K \sqrt{n}}{6 K \b_0 \sqrt{n}} = \frac{6}{\b_0} \le 6D_0.
$$
So, for every $u_0 \in E$ we can choose $v_0 \in \MM$ such that $\|u_0-v_0\|_2 \le K \b_0 \sqrt{n}$. 
Since $Ax_0 = u_0$, it follows that $\|Ax_0 - v_0\|_2 \le K \b_0 \sqrt{n}$.
We have shown that
$$
\EE_1 \cap \EE_0^c \cap \EE_K
\subseteq \big\{ \exists x_0 \in S_{D_0}, \, \exists v_0 \in \MM :\, \|Ax_0 - v_0\|_2 \le K \b_0 \sqrt{n} \wedge \EE_K \}.
$$
An application of Lemma~\ref{sbp Ax sublevel} and a union bound over $v_0 \in \MM$ give
$$
\P(\EE_1 \cap \EE_0^c \cap \EE_K)
\le |\MM| \cdot n^{-c'n}
\le 6D_0 \cdot n^{-c'n}
\le n^{-c'n/2}
$$
where we used that $D_0 \le n^{c/\l}$, and since we can assume that constant $c>0$ appropriately small. 
The proof of Structure Theorem~\ref{structure} is complete.
\end{proof}

\section{Small ball probabilities for quadratic forms}

Now that we developed a machinery for estimating small ball probabilities, 
we can come back to our main task, estimating the small ball probability 
for quadratic forms. Recall that by Proposition~\ref{dist via quadratic} 
and Remark~\ref{A versus B}, the distance problem reduces to estimating
L\'evy concentration function for the self-normalized quadratic forms:  
\begin{equation}							\label{sbp quadratic form question}
\LL \Big( \frac{\< A^{-1} X, X\> }{\sqrt{1 + \|A^{-1}X\|_2^2}}, \, \e \Big) \le ?
\end{equation}
Here and throughout this section, $A$ denotes the $n \times n$ symmetric random matrix 
satisfying assumptions \eqref{aij}.  $X$ denotes a random vector whose 
entries are independent of $A$ and of each other, identically distributed, 
and satisfy the same moment assumptions \eqref{aij} as those of $A$, namely
they have zero mean, unit variance, and fourth moment bounded by $M_4^4$. 

The goal of this section is to prove the following estimate. 

\begin{theorem}[Small ball probabilities for quadratic forms]				\label{sbp quadratic form}
  Let $A$ be an $n \times n$ random matrix which satisfies \A, 
  and let $X$ be a random vector in $\R^n$ whose entries 
  are independent of each other and of $A$,  identically distributed, 
  and satisfy the same moment assumptions \eqref{aij} as those of $A$, namely
  they have zero mean, unit variance, and fourth moment bounded by $M_4^4$. 
  There exist constants $\Const{sbp quadratic form}, \const{sbp quadratic form} > 0$ 
  that depend only on the parameters $K$ and $M_4$ 
  from the assumptions \eqref{norm}, \eqref{aij}, and such that the following holds. 
  For every $\e \ge 0$ and every $u \in \R$, one has
  \begin{equation}							\label{eq sbp quadratic form}
  \P \Big\{ \frac{|\< A^{-1} X, X\> - u| }{\sqrt{1 + \|A^{-1}X\|_2^2}} \le \e \wedge \EE_K \Big\} 
  \le \Const{sbp quadratic form} \e^{1/9} + 2\exp(-n^{\const{sbp quadratic form}}).
  \end{equation}
\end{theorem}

In particular, we have a desired bound for L\'evy concentration function in 
\eqref{sbp quadratic form question}, namely 
$\Const{sbp quadratic form} \e^{1/9} + 2\exp(-n^{\const{sbp quadratic form}}) + \P(\EE_K^c)$.

To prove Theorem~\ref{sbp quadratic form}, we will first decouple the enumerator $\< A^{-1} X, X\> $
from the denominator $\sqrt{1 + \|A^{-1}X\|_2^2}$ by showing that $\|A^{-1}X\|_2 \sim \|A^{-1}\|_\HS$ with high probability. 
This is done in Section~\ref{s: size of inverse}. Then we decouple the quadratic form $\< A^{-1} X, X\> $. 
An ideal decoupling argument would replace $\< A^{-1} X, X\> $ by $\< A^{-1} X, X' \> $ where $X'$ is independent random 
vector; our argument will be of similar nature. Then by conditioning on $X$ we obtain a linear form, 
and we can estimate its small ball probabilities using the Littlewood-Offord theory (specifically, using 
Proposition~\ref{sbp via reg lcd} and Structure Theorem~\ref{structure}). This will be done in 
Section~\ref{s: sbp quadratic form proof}.

\subsection{Size of $A^{-1} X$}			\label{s: size of inverse}

The following result compares the size of the denominator $\sqrt{1 + \|A^{-1}X\|_2^2}$ 
appearing in \eqref{eq sbp quadratic form} to $\|A^{-1}\|_\HS$.

\begin{proposition}[Size of $A^{-1} X$]			\label{size of inverse}
  Let $A$ be a random matrix which satisfies \A.
  There exist constants $c, \Const{size of inverse}, \const{size of inverse} > 0$
  that depend only on the parameters $K$ and $M_4$ 
  from the assumptions \eqref{norm}, \eqref{aij}, and such that the following holds. 
  Let $n^{-c} \le \l \le c$. 
  The random matrix $A$ has the following property with probability at least $1 - e^{-cn}$. 
  If $\EE_K$ holds, then for every $\e > 0$ one has: 
  \begin{enumerate}[(i)]
    \item with probability at least $1-e^{-\const{size of inverse}n}$ in $X$, we have 
      $$
      \|A^{-1}X\|_2 \ge \Const{size of inverse};
      $$
    \item with probability at least $1-\e$ in $X$, we have 
      $$
      \|A^{-1}X\|_2 \le \e^{-1/2} \|A^{-1}\|_\HS;
      $$
    \item with probability at least $1 - \Const{size of inverse}\e/\sqrt{\l} - n^{-\const{size of inverse}/\l}$ in $X$, 
      we have
      $$
      \|A^{-1}X\|_2 \ge \e \|A^{-1}\|_\HS.
      $$
  \end{enumerate}
\end{proposition}

The proof of this result uses the following elementary lemma. 

\begin{lemma}[Sums of dependent random variables]			\label{sum dependent}
  Let $Z_1, \ldots, Z_n$ be arbitrary non-negative random variables (not necessarily independent), 
  and $p_1,\ldots, p_n$ be non-negative numbers such that
  $$
  \sum_{k=1}^n p_k = 1. 
  $$
  Then for every $\e \in \R$ one has
  $$
  \P \Big\{ \sum_{k=1}^n p_k Z_k \le \e \Big\} 
  < 2 \sum_{k=1}^n  p_k \, \P \{ Z_k \le 2\e \}.
  $$
\end{lemma}

\begin{proof}
By Markov's inequality, the event $\sum_{k=1}^n p_k Z_k \le \e$
implies $\sum_k p_k \one_{\{Z_k > 2\e\}} < 1/2$
and, consequently, $\sum_k p_k \one_{\{Z_k \le 2\e\}} > 1/2$.
Therefore, 
\begin{align*}
\P \Big\{ \sum_{k=1}^n p_k Z_k \le \e \Big\} 
  &\le \P \Big\{ \sum_k p_k \one_{\{Z_k \le 2\e\}} > 1/2 \Big\} \\
  &< 2 \E \sum_k p_k \one_{\{Z_k \le 2\e\}} \quad \text{(again by Markov's inequality)} \\
  &= 2 \sum_{k=1}^n  p_k \, \P \{ Z_k \le 2\e \}.
\end{align*}
The proof is complete.
\end{proof}

\begin{proof}[Proof of Proposition~\ref{size of inverse}.]
Let $e_1,\ldots,e_n$ denotes the canonical basis of $\R^n$, and let 
$$
x_k := \frac{A^{-1}e_k}{\|A^{-1}e_k\|_2}, \quad k=1,\ldots,n.
$$
Let us apply Structure Theorem~\ref{structure} combined with the union bound over $k=1,\ldots,n$.
We do this with $L = L_0$ a suitably large constant depending on parameters $K$ and $M$ only
(chosen so that Proposition~\ref{sbp via reg lcd} can be applied below). 
We see that the random matrix $A$ has the following property with probability 
at least $1 - n \cdot 2e^{-\const{structure}n} \ge 1 - 2e^{-\const{structure}n/2}$: if $\EE_K$ holds then 
\begin{equation}										\label{D event}
x_k \in \Incomp(c_0,c_1),\,  \Dhat_L(x_k, \l) \ge L^{-2} n^{\const{structure}/\l}
\quad k=1,\ldots,n.
\end{equation}
Let us fix a realization of $A$ with this property. We shall deduce properties (i), (ii), (iii) from it. 
Without loss of generality we may assume that $\EE_K$ holds. 

{\em (i)}
We have
$$
\|X\|_2 \le \|A\| \|A^{-1}X\|_2.
$$
By $\EE_K$, we have $\|A\| \le 3 K \sqrt{n}$. 
Moreover, Lemma~\ref{Levy weak} and Tensorization Lemma~\ref{tensorization}
imply that the random vector $X$ satisfies
$\|X\|_2 \ge c' \sqrt{n}$ with probability at least $1 - e^{-c'n}$, for some constant $c' = c'(K,M)>0$.
It follows that $\|A^{-1}X\|_2 \ge c'/3K$ with the same probability, 
so part (i) of the proposition is proved.

{\em (ii)}
Using that $A$ is a symmetric matrix, we express
\begin{equation}										\label{norm sum}
\|A^{-1}X\|_2^2 
= \sum_{k=1}^n \< A^{-1}X, e_k\> ^2
= \sum_{k=1}^n \< A^{-1}e_k, X \> ^2
= \sum_{k=1}^n \|A^{-1} e_k\|_2^2 \, \< x_k, X\> ^2.
\end{equation}
Recall that the coordinates of $X$ are independent random variables with zero mean and unit variance. 
Therefore $\E_X \< x_k, X\> ^2 = 1$ for all $k$, so 
$$
\E_X \|A^{-1}X\|_2^2 
= \sum_{k=1}^n \|A^{-1} e_k\|_2^2
= \|A^{-1}\|_\HS^2.
$$
An application of Markov's inequality yields part (ii) of the proposition.

{\em (iii)} 
Fix $k \le n$. Then $\< x_k, X\> $ can be expressed as a sum of independent 
random variables $\sum_{i=1}^n x_{ki} X_i$, where $x_{ki}$ and $X_i$ denote the coordinates
of $x_k$ and of $X$ respectively. This sum can be estimated using 
Proposition~\ref{sbp via reg lcd} (with $J=[n]$) combined with the estimate
\eqref{D event} on the regularized LCD of $x_k$. 
This gives 
\begin{equation}										\label{sbp individual}
\P_X \big\{ |\< x_k, X\> | \le \sqrt{2} \, \e \big\}
\le \Const{sbp via reg lcd} L \Big( \frac{\e}{\sqrt{\l}} + L^2 n^{-\const{structure}/\l} \Big).
\end{equation}
Now we combine these estimates for all $k$ using \eqref{norm sum} 
and Lemma~\ref{sum dependent} with $p_k = \|A^{-1} e_k\|_2^2 / \|A^{-1}\|_\HS^2$;
note that $\sum_{k=1}^n p_k = 1$. 
We obtain
\begin{align*}
\P_X \big\{ \|A^{-1}X\|_2 \le \e \|A^{-1}\|_\HS \big\} 
  &= \P \Big\{ \sum_{k=1}^n p_k \< x_k, X\> ^2 \le \e^2 \Big\} \\
  &\le 2 \sum_{k=1}^n p_k \, \P \big\{ \< x_k, X\> ^2 \le 2\e^2 \big\} 
  		\quad \text{(by Lemma~\ref{sum dependent})} \\
  &\le 2 \Const{sbp via reg lcd} L \Big( \frac{\e}{\sqrt{\l}} + L^2 n^{-\const{structure}/\l} \Big) 
  		\quad \text{(by \eqref{sbp individual}).}
\end{align*}
This proves part (iii), and completes the proof of Proposition~\ref{size of inverse}.
\end{proof}

\subsection{Decoupling quadratic forms}				\label{s: decoupling}

Decoupling the quadratic form $\< A^{-1} X, X\> $ is based on the following general result.
Similar decoupling techniques for quadratic forms were first applied by G\"otze \cite{Gotze}
and used in literature many times since then; in particular such a decoupling argument was used 
in \cite{CTV, Costello} in a context similar to ours.

\begin{lemma}[Decoupling quadratic forms]			\label{decoupling}
  Let $G$ be an arbitrary symmetric $n \times n$ matrix, and let $X$ be a random vector in $\R^n$
  with independent coordinates. Let $X'$ denote an independent copy of $X$. 
  Consider a subset $J \subseteq [n]$. Then for every $\e \ge 0$ one has
  \begin{align*}
  \LL \big( \< GX,X\> , \e \big)^2
    &= \sup_{u \in \R} \P \big\{ |\< GX,X\> - u| \le \e \big\}^2 \\
    &\le \P_{X,X'} \Big\{  \big| \< G(P_{J^c}(X-X')), P_J X\> - v \big| \le \e \Big\}
\end{align*}
  where $v$ is some random variable whose value is determined by 
  the $J^c \times J^c$ minor of $G$ and the random vectors $P_{J^c} X$, $P_{J^c} X'$. 
\end{lemma}

The point of this result is that, upon conditioning on the coordinates of $X$ and $X'$ in $J^c$, 
the vectors $x_0 := G(P_J^c(X-X'))$ and $v$ become fixed. So the L\'evy concentration 
function of the quadratic form $\< GX, X\> $ gets bounded by 
the L\'evy concentration function of the linear form 
$\< x_0, P_J X\> $. The latter, as we know, can be estimated using the Littlewood-Offord theory. 

The proof of Lemma~\ref{decoupling} is based on the general decoupling lemma from \cite{Sidorenko}, 
which was already used for a purpose similar to ours in \cite{Costello}. 

\begin{lemma}				\label{sidorenko}
  Let $Y$ and $Z$ be independent random variables or vectors, 
  and let $Z'$ be an independent copy of $Z$. 
  Let $\EE(Y,Z)$ be an event which is determined by the values of $Y$ and $Z$. 
  Then 
  $$
  \P \big\{ \EE(Y,Z) \big\}^2 \le \P \big\{ \EE(Y,Z) \cap \EE(Y,Z') \big\}. \qquad \qed
  $$
\end{lemma}

\begin{proof}[Proof of Decoupling Lemma~\ref{decoupling}.]
By permuting the coordinates, without loss of generality we can assume that 
$J$ and $J^c$ are intervals of coordinates with $\sup J \le \inf J^c$.
The decomposition $[n] = J \cup J^c$ induces the decomposition of the matrix $A$ 
and all the vectors in question, 
$$
G = 
\begin{pmatrix} 
  E & F \\ 
  F^* & H 
\end{pmatrix}, \quad
X = \begin{pmatrix} Y \\ Z \end{pmatrix}, \quad
X' = \begin{pmatrix} Y' \\ Z' \end{pmatrix}; \quad
\text{let } \widetilde{X} = \begin{pmatrix} Y \\ Z' \end{pmatrix}.
$$ 
Here $E$ is a $J \times J$ minor of $G$, $H$ is a $J \times J^c$ minor, etc., 
and similarly $Y \in \R^J$, $Z \in \R^{J^c}$, etc.
Let us fix a $u \in \R$ and apply Lemma~\ref{sidorenko}; this gives
\begin{equation}							\label{p square}
p^2 := \P \big\{ |\< GX,X\> - u| \le \e \big\}^2
\le \P_{X, \widetilde{X}} \big\{ |\< GX,X\> - u| \le \e \wedge |\< G \widetilde{X}, \widetilde{X} \> - u| \le \e \big\}.
\end{equation}
By the triangle inequality, 
$$
p^2 \le \P_{X, \widetilde{X}} \big\{ |\< GX,X\> - \< G \widetilde{X}, \widetilde{X} \> | \le 2\e \big\}.
$$
By our decomposition, we have 
\begin{gather*}
\< GX,X\> = \< EY,Y\> + 2 \< FZ,Y\> + \< HZ,Z\> , \\
\< G\widetilde{X},\widetilde{X}\> = \< EY,Y\> + 2 \< FZ',Y\> + \< HZ',Z'\> .
\end{gather*}
Hence
$$
\< GX,X\> - \< G \widetilde{X}, \widetilde{X} \> = 2 \< F(Z-Z'), Y\> + \< HZ,Z\> - \< HZ',Z'\> .
$$
Recall that $F$ is the restriction of the matrix $G$ onto the pairs of coordinates in $J \times J^c$, 
that $Z - Z'$ is the restriction of the vector $X-X'$ onto the coordinates in $J^c$, 
and that $Y$ is the restriction of $X$ onto the coordinates in $J$.
So 
$$
\< F(Z-Z'), Y\> = \< G(P_{J^c}(X-X')), P_J X\> .
$$ 
Similarly we can see that the value of $\< HZ,Z\> - \< HZ',Z'\> $ depends on the $J^c \times J^c$ minor $H$
and on the restrictions of $X$ and $X'$ onto the coordinates in $J^c$. 
So setting $v = 2 \< HZ,Z\> - 2 \< HZ',Z'\> $, we express 
$$
\< GX,X\> - \< G \widetilde{X}, \widetilde{X} \> = 2 \< G(P_{J^c}(X-X')), P_J X\> + v.
$$
This and \eqref{p square} completes the proof of Decoupling Lemma~\ref{decoupling}.
\end{proof}

\subsection{Proof of Theorem~\ref{sbp quadratic form}} 		\label{s: sbp quadratic form proof}

Our argument will be based on decoupling the quadratic form $\< AX, X\> $, and treating 
the resulting linear form using the Littlewood-Offord theory developed earlier in this paper. 

\paragraph{Step 1: Constructing a random subset $J$ and assignment $\spread(x)$.}
The decoupling starts by decomposing $[n]$ into two random sets $J$ and $J^c$. 
To this end, we consider independent $\{0,1\}$-valued random variables $\d_1,\ldots,\d_n$
(``selectors'') with $\E \d_i = c_{oo}/2$. (Recall that the constant $c_{oo}$, which depends on $K$ and $M_4$ only, was fixed 
in the definition of the regularized LCD in Section~\ref{s: reg LCD}.)
We then define
$$
J := \{ i \in [n] :\, \d_i = 0 \}.
$$ 
Then $\E|J^c| = c_{oo} n/2$. 
By a basic result in large deviations (see e.g. \cite{AS} Theorem~A.1.4), the bound
\begin{equation}							\label{c}
  |J^c| \le c_{oo} n
\end{equation}
holds with high probability:
$$
\P_J \{ \text{\eqref{c} holds} \} \ge 1 - 2 e^{-c'_{oo} n}
$$
where $c'_{oo} = c_{oo}^2/2$.

Consider a fixed realization of $J$ that satisfies \eqref{c}.
As we know from Section~\ref{s: reg LCD}, at least $2c_{oo} n$ coordinates of a vector $x \in \Incomp(c_0,c_1)$
satisfy the regularity condition \eqref{eq spread}. It follows that for each vector $x \in \Incomp(c_0,c_1)$ we can assign 
a subset 
\begin{equation}							\label{spread in J}
\spread(x) \subseteq J, \quad |\spread(x)| = \lceil c_{oo} n \rceil
\end{equation}
and so that the regularity condition \eqref{eq spread} holds for all $k \in \spread(x)$.
If there is more than one way to assign $\spread(x)$ to $x$, we choose one fixed way to do so. 
This results in a valid assignment (per Section~\ref{s: reg LCD}) that depends only on the choice of the random set $J$.
We shall use this assignment in applications of Definition~\ref{def reg LCD} of the regularized LCD of $x$. 

\paragraph{Step 2: Estimating the denominator $\sqrt{1 + \|A^{-1}X\|_2^2}$ and LCD of the inverse.}
Lemma~\ref{size of inverse} will allow us to replace in \eqref{eq sbp quadratic form} 
the denominator $\sqrt{1 + \|A^{-1}X\|_2^2}$ by $\|A^{-1}\|_2^2$.
However, we have to do this carefully in order to withstand losses that will occur at the decoupling step.
So, let $\e_0 \in (0,1)$ and let $X'$ denote an independent copy of the random vector $X$.
We consider the following event that is determined by the random matrix $A$, 
random vectors $X, X'$ and the random set $J$: 
\begin{equation}							\label{a}
\sqrt{\e_0} \sqrt{1 + \|A^{-1}X\|_2^2} 
\le \|A^{-1}\|_\HS 
\le \frac{1}{\e_0} \|A^{-1}(P_{J^c}(X - X'))\|_2.
\end{equation}
Recall that the coordinates $X_i$ of $X$ are independent random variables with zero mean, unit variance, and
$\E X_i^4 \le M_4^4$. It follows that the coordinates $Y_i = \d_i (X_i - X_i')$ of the 
random vector $Y := P_{J^c}(X - X')$ are again independent random variables with 
$$
\E Y_i = 0, \quad \E Y_i^2 = c_{oo}, \quad \E Y_i^4 \le 8 c_{oo} M_4^4.
$$
We see that Proposition~\ref{size of inverse} applies for $X$, and also for $X$ replaced by $c_{oo}^{-1/2} X$
(in the latter case with $M_4$ replaced by $2 c_{oo}^{-1/4} M_4$). It follows that 
$$
\P_{A,X,X',J} \{ \text{\eqref{c} holds} \vee \EE_K^c\}
\ge 1 - \frac{C' \e_0}{\sqrt{\l}} - n^{-c'/\l} - 2e^{-c'n}
$$
where $C', c'>0$ depend only on $K$ and $M_4$. 

Consider the random vector 
\begin{equation}							\label{x0}
x_0 := \frac{A^{-1}(P_{J^c}(X - X'))}{\|A^{-1}(P_{J^c}(X - X'))\|_2}.
\end{equation}
(If the denominator equals zero, assign to $x_0$ an arbitrary fixed vector in $S^{n-1}$.)
Let us condition on an arbitrary realization of random vectors $X, X'$ and on a realization 
of $J$ which satisfies \eqref{c}. 
Fix some value of the parameter $\l$ satisfying $n^{-\const{structure}} \le \l \le c_{oo}/3$ 
as required in Structure Theorem~\ref{structure}, and consider the event 
\begin{equation}							\label{b}
x_0 \in \Incomp(c_0,c_1) \quad \text{and} \quad \Dhat_{L_0}(x_0,\l) \ge C'' n^{c''/\l},
\end{equation}
which depends on the random matrix $A$. 
By Structure Theorem~\ref{structure}, the conditional probability is 
$$
\P_A \big\{ \text{\eqref{b} holds} \vee \EE_K^c \,|\, X, X', J \text{ satisfies \eqref{c}} \big\} 
\ge 1 - 2 e^{-c''n}.
$$
Here $L_0, C'', c''>0$ depend only on $K$ and $M_4$.

Combining the three probabilities, we obtain
\begin{align}
\P_{A,X,X',J} &\big\{ (\text{\eqref{c}, \eqref{a}, \eqref{b} hold}) \vee \EE_K^c \big\} \nonumber\\
  &\ge 1 - 2 e^{-c'_{oo} n} - \Big(\frac{C' \e_0}{\sqrt{\l}} + n^{-c'/\l} + 2e^{-c'n} \Big) - 2 e^{-c''n} \nonumber\\
  &=: 1-p_0.				\label{p0}
\end{align}
It follows that there exists a realization of $J$ that satisfies \eqref{c} and such that 
$$
\P_{A,X,X'} \big\{ (\text{\eqref{a}, \eqref{b} hold}) \vee \EE_K^c \big\} 
\ge 1-p_0.
$$
Let us fix such a realization of $J$ for the rest of the proof. 
An application of Fubini's theorem shows that the random matrix $A$ has the following 
property with probability at least $1-\sqrt{p_0}$:
$$
\P_{X,X'} \big\{ (\text{\eqref{a}, \eqref{b} hold}) \vee \EE_K^c \,|\, A \big\} 
\ge 1-\sqrt{p_0}.
$$
But the event $\EE_K^c$ depends on $A$ only and not on $X$ or $X'$. 
Therefore, the random matrix $A$ has the following 
property with probability at least $1-\sqrt{p_0}$.
Either $\EE_K^c$ holds, or:
\begin{equation}							\label{d}
\EE_K \text{ holds  and } \P_{X,X'} \big\{ \text{\eqref{a}, \eqref{b} hold} \,|\, A \big\} 
\ge 1-\sqrt{p_0}.
\end{equation}

\paragraph{Step 3: decoupling.}
The event we are interested in is
$$
\EE := \left\{ \frac{|\< A^{-1} X, X\> - u| }{\sqrt{1 + \|A^{-1}X\|_2^2}} \le \e \right\}.
$$
We need to estimate the probability 
$$
\P_{A,X} (\EE \cap \EE_K)
  \le \P_{A,X} \{ \EE \wedge \text{\eqref{d} holds} \} + \P_{A,X} \{ \EE_K \wedge \text{\eqref{d} fails} \}.
$$
By the previous step in the proof, the second term here is bounded by $\sqrt{p_0}$.
Therefore
$$
\P_{A,X} (\EE \cap \EE_K)
  \le \sup_{A \text{ satisfies \eqref{d}}} \P_X(\EE \,|\, A) + \sqrt{p_0}.
$$
Computing the same probability in the larger space determined by the random vectors $X,X'$, 
and using property \eqref{d}, we write
\begin{equation}							\label{desired event split}
\P_{A,X} (\EE \cap \EE_K)
  \le \sup_{A \text{ satisfies \eqref{d}}} \P_{X,X'} \big\{ \EE \wedge \text{\eqref{a} holds} \,|\, A \big\} + 2\sqrt{p_0}. 
\end{equation}
Let us fix a realization of a random matrix $A$ satisfying \eqref{d} for the rest of the proof. 
So our goal is to bound the probability
$$
p_1 := \P_{X,X'} \big\{ \EE \wedge \text{\eqref{a} holds} \big\}.
$$
Using definition of $\EE$ and the first inequality in property \eqref{a}, we have
$$
p_1 \le P_{X,X'} \Big\{ |\< A^{-1} X, X\> - u| \le \frac{\e}{\sqrt{\e_0}} \|A^{-1}\|_\HS \Big\}.
$$
We apply Decoupling Lemma~\ref{decoupling}, and obtain
$$
p_1^2 \le \P_{X,X'} \{ \EE_0 \}
$$
where
$$
\EE_0 = \Big\{ \big| \< A^{-1}(P_{J^c}(X-X')), P_J X\> - v \big| 
  \le \frac{\e}{\sqrt{\e_0}} \|A^{-1}\|_\HS \Big\}
$$
and where $v =v(A^{-1}, P_{J^c}X, P_{J^c}X')$ denotes a number that depends on $A^{-1}$, 
$P_{J^c}X$, $P_{J^c}X'$ only.
Further, using property \eqref{d} (in which conditioning on $A$ is no longer needed as we are treating $A$ as a fixed matrix), 
we get 
$$
p_1^2 \le \P_{X,X'} \{ \EE_0 \} 
\le \P_{X,X'} \big\{ \EE_0 \wedge \text{\eqref{a}, \eqref{b} hold} \big\} + \sqrt{p_0}.
$$
Let us divide both sides in the inequality defining the event $\EE_0$ by $\|A^{-1}(P_{J^c}(X-X'))\|_2$.
Using definition \eqref{x0} of $x_0$ and the second inequality in \eqref{a}, we obtain 
\begin{equation}							\label{decoupling complete}
p_1^2 \le \P_{X,X'} \Big\{ \big| \< x_0, P_J X\> - w \big|
  \le \e_0^{-3/2} \e \wedge \text{\eqref{b} holds} \Big\} + \sqrt{p_0}
\end{equation}
where $w = w(A^{-1}, P_{J^c}X, P_{J^c}X')$ is an appropriate number.

\paragraph{Step 4: The small ball probabilities of a linear form.}
By definition, the random vector $x_0$ is determined by the random vector $P_{J^c}(X-X')$, 
which is independent of the random vector $P_J X$.
So if we fix an arbitrary realization of the random vectors $P_{J^c} X$ and $P_{J^c} X'$, 
this will fix the vector $x_0$ and the number $w$ in \eqref{decoupling complete}.
Since moreover \eqref{b} is a property of $x_0$, we conclude that 
$$
p_1^2 \le \sup_{\substack{x_0 \text{ satisfies \eqref{b}} \\ w \in \R}}
  \P_{P_J X} \Big\{ \big| \< x_0, P_J X\> - w \big| \le\e_0^{-3/2} \e \Big\} + \sqrt{p_0}.
$$
So let us fix a vector $x_0 = (x_{01}, \ldots, x_{0n}) \in S^{n-1}$ that satisfies \eqref{b} and a number $w \in \R$.
We have reduced the problem to estimating the small ball probabilities for the sum of independent random variables 
$$
\< x_0, P_J X\> = \sum_{k \in J} x_{0k} \xi_k
$$
where we denote $X = (\xi_1, \ldots, \xi_n)$. 

We can apply Proposition~\ref{sbp via reg lcd} for this sum, noting that by \eqref{spread in J} we have
$J \supseteq \spread(x_0) \supseteq I(x)$ as required there. (The last inclusion follows by the definition of the maximizing 
set $I(x)$, recall Definition~\ref{def reg LCD}.) It follows that 
$$
P_{P_J X} \Big\{ \big| \< x_0, P_J X\> - w \big| \le \e_0^{-3/2} \e \Big\}
  \le \frac{C_1 \e_0^{-3/2} \e}{\sqrt{\l}} + \frac{C_1}{\Dhat_{L_0}(x_0,\l)},
$$
for some $C_1 = C_1(K, M_4)$.
Using property \eqref{b} to bound the second term in the right hand side, we obtain
$$
p_1^2 \le \frac{C_1 \e_0^{-3/2} \e}{\sqrt{\l}} + C'_1 n^{-c''/\l} + \sqrt{p_0},
$$
Now we estimate the probability of the desired event in \eqref{desired event split} as 
\begin{align*}
\P_{A,X} (\EE \cap \EE_K)
  &\le p_1 + 2 \sqrt{p_0} \\
  &\le \Big( \frac{C_1 \e_0^{-3/2} \e}{\sqrt{\l}} \Big)^{1/2} + \big( C'_1 n^{-c''/\l} \big)^{1/2} + p_0^{1/4} + 2 \sqrt{p_0}.
\end{align*}
Recalling the definition \eqref{p0} of $p_0$ and simplifying, we obtain 
$$
\P_{A,X} (\EE \cap \EE_K)
  \le \Big( \frac{C_1 \e_0^{-3/2} \e}{\sqrt{\l}} \Big)^{1/2} 
    + \Big( \frac{C' \e_0}{\sqrt{\l}} \Big)^{1/4} + C'_1 n^{-c'_1/\l} + C'_1 e^{-c'_1 n}.
$$

\paragraph{Step 5: Optimizing the parameters.}
This inequality holds for all $\e_0 > 0$, so we can optimize in $\e_0$. 
Setting $\e_0 = \e^{1/2} / \l^{1/8}$, we obtain after some simplification that 
$$
\P_{A,X} (\EE \cap \EE_K)
  \le \frac{C_2 \e^{1/8}}{\l^{5/32}} + C'_1 n^{-c'_1/\l} + C'_1 e^{-c'_1 n}.
$$
By assumption, $\l \ge n^{-\const{structure}}$ where $\const{structure}>0$ is a small constant.
So, for appropriately chosen constants,
the term $n^{-c'_1/\l}$ dominates the term $e^{-c'_1 n}$. We obtain 
$$
\P_{A,X} (\EE \cap \EE_K)
  \le \frac{C_2 \e^{1/8}}{\l^{5/32}} + 2C'_1 n^{-c'_1/\l}.
$$
Recall that this inequality holds for all $\e \ge 0$ and $n^{-\const{structure}} \le \l \le c_{oo}/3$, 
so we can also optimize in $\l$. 
For convenience, we isolate this step as a separate elementary observation.

\begin{fact}[Optimization]					\label{optimization}
  Let $C \ge 1$, $a,b,c',c > 0$. There exists numbers $C_0$ and $n_0$ that depend only 
  on $a,b,c',C,c$ and such that the following holds. Let $n \ge n_0$. 
  Consider a function $p(\e) : [0,1] \to \R_+$ which satisfies
  $$
  p(\e) \le M^a \e^b + n^{-c'M}
  \quad \text{for all } \e \in [0,1] \text{ and } C \le M \le n^{c}.
  $$
  Then 
  $$
  p(\e) \le C_0 \e^{b-0.01} + n^{-c'n^c}
  \quad \text{for all } \e \in [0,1].
  $$
\end{fact}

\begin{proof}[Proof of Fact~\ref{optimization}.]
Choose some number $C \le M_0 \le n^{c}$ whose value will be determined later. 
By the assumption, the inequality
\begin{equation}							\label{p opt 1}
p \le M_0^a \e^b + n^{-c'M_0} \le (M_0^a + 1) \e^b
\end{equation}
holds for all $\e \ge n^{-c' M_0 / b}$.
On the other hand, using the assumption with $M = n^c$, we see that the inequality 
$$
p \le n^{ac} \e^b + n^{-c'n^c} \le \e^{b-0.01} + n^{-c'n^c}
$$
holds for all $\e \le n^{-100 a c}$. 
Let us choose $M_0$ as the minimal number such that $M_0 \ge C$ and $c'M_0/b \ge 100 a c$. 
Note that we have $C \le M \le n^c$ as required, for sufficiently large $n_0$. 
Therefore, every $\e$ belongs to the range where inequality \eqref{p opt 1} holds or \eqref{p opt 2} holds, or both. 
So at least one of these inequalities holds for all $\e \ge 0$. 
This completes the proof with $C_0 = M_0^a + 1$.
\end{proof}

Applying Fact~\ref{optimization} with $M = 1/\l$, $a = 5/32$ and $b = 1/8$, we conclude that 
\begin{equation}							\label{p opt 2}
\P_{A,X} (\EE \cap \EE_K)
  \le C_0 \e^{1/9} + n^{-c'n^c}
\end{equation}
holds for all $\e \in [0,1]$, where $c=\const{structure}$. 
Since we can choose $C_0 \ge 1$, the same inequality trivially holds 
for $\e > 1$ as the right hand side becomes larger than $1$. 
The proof of Theorem~\ref{sbp quadratic form} is complete.
\qed

\section{Consequences: the distance problem and invertibility of random matrices}

\subsection{The distance theorem}

An application of Theorem~\ref{sbp quadratic form} together with Proposition~\ref{dist via quadratic} 
produces a satisfactory 
solution to the distance problem posed in the beginning of Section~\ref{s: distance via sbp}.

\begin{corollary}[Distance between random vectors and subspaces]		\label{distance}
  Let $A$ be a random matrix satisfying \A.
  There exist constants $C,c > 0$ that depend only on the parameters $K$ and $M_4$ 
  from \eqref{norm}, \eqref{aij}, and such that the following holds. 
  Let $A_k$ denote the $k$-th column of $A$ and $H_k$ denote the span of the other columns.  
  For every $\e \ge 0$, one has
  $$
  \P \big\{ \dist(A_k, H_k) \le \e \wedge \EE_K \big\} 
    \le \Const{distance} \e^{1/9} + 2\exp(-n^{\const{distance}}).
  $$
\end{corollary}

\begin{proof}
By permuting the coordinates, we can assume without loss of generality that $k=1$. 
Proposition~\ref{dist via quadratic} states that 
$$
\dist(A_1, H_1) = \frac{\big|\< B^{-1} X, X\> - a_{11} \big|}{\sqrt{1 + \|B^{-1}X\|_2^2}}.  
$$
where $B$ denotes the $(n-1) \times (n-1)$ minor of $A$ obtained by removing the first row and
the first column from $A$ and $X \in \R^{n-1}$ denotes the first column of $A$ with the first entry removed.
By assumptions, $B$ is a random matrix which satisfies the same assumptions \A
as $A$ (except the dimension is one less), and $X$ is an independent random vector whose entries 
also satisfy the same assumptions \eqref{aij}.
So we can apply Theorem~\ref{sbp quadratic form} for $B$ and $X$. 
Conditioning on the independent entry $a_{11} = u$, we obtain that 
$$
\P \Big\{ \frac{\big|\< B^{-1} X, X\> - a_{11} \big|}{\sqrt{1 + \|B^{-1}X\|_2^2}} \le \e \wedge \EE_K \Big\} 
  \le \Const{sbp quadratic form} \e^{1/9} + 2\exp(-(n-1)^{\const{sbp quadratic form}}).
$$
This completes the proof. 
\end{proof}

\subsection{Invertibility of random matrices: proof of Theorem~\ref{delocalization}.}

We can now derive the main result of the paper, Theorem~\ref{delocalization}.
In Section~\ref{s: reductions}, we reduced the problem to proving the invertibility bound \eqref{delocalization goal}. 
We shall now establish this bound, which immediately implies Theorem~\ref{delocalization}.

\begin{theorem}[Invertibility of symmetric random matrices]				\label{invertibility}
  Let $A$ be a random matrix which satisfies \A. 
  Consider a number $K > 0$. 
  Then, for all $\e \ge 0$, one has 
  $$
  \P \Big\{ \min_k |\l_k(A)| \le \e n^{-1/2} 
  \wedge \|A\| \le 3K \Big\} \le C \e^{1/9} + 2\exp(-n^c),
  $$  
  where $C, c > 0$ depend only on the fourth moment bound $M_4$ from \eqref{aij} and on $K$.
\end{theorem}

\begin{proof}
Denote by $p$ the probability in question. As we observed in Section~\ref{s: reductions}, 
$$
p = \P \Big\{ \min_{x \in S^{n-1}} \|Ax\|_2 \le \e n^{-1/2} 
  \wedge \EE_K \Big\}.
$$
In \eqref{split}, we split the invertibility problem into two, for compressible and incompressible vectors: 
\begin{multline*}	
p \le \P \Big\{ \inf_{x \in \Comp(c_0,c_1)} \|Ax\|_2 \le \e n^{-1/2} \wedge \EE_K \Big\} \\
+ \P \Big\{ \inf_{x \in \Incomp(c_0,c_1)} \|Ax\|_2 \le \e n^{-1/2} \wedge \EE_K \Big\}.
\end{multline*}
The values of $c_0,c_1$ were then fixed in Remark~\ref{c0 c1}. 
The probability for the compressible vectors is bounded by $2 e^{-\const{sbp comp} n}$ by \eqref{compressible solved}. 
The probability for the incompressible vectors is estimated via distances in Lemma~\ref{l: via distance}, 
see Remark~\ref{adding EEK}. This gives
$$
p \le 2 e^{-\const{sbp comp} n} 
  + \frac{1}{c_0 n} \sum_{k=1}^n \P \big\{ \dist( A_k, H_k) \le c_1^{-1} \e \wedge \EE_K \big\}.
$$
Finally, the distances are estimated in Corollary~\ref{distance}, which gives 
$$
p \le 2 e^{-\const{sbp comp} n} + \Const{distance} \e^{1/9} + 2\exp(-n^{\const{distance}}).
$$
Choosing the values of the constant $c>0$ sufficiently small, we complete the proof of Theorem~\ref{invertibility}. 
\end{proof}


\begin{thebibliography}{99}

\bibitem{AS} N. Alon, J. Spencer, 
  {\em The probabilistic method}, 
  Second edition. Wiley-Interscience, New York, 2000. 
  
\bibitem{BVW} J. Bourgain, V. Vu, P. Wood, 
  {\em On the singularity probability of discrete random matrices},
  J. Funct. Anal. 258 (2010), 559--603.

\bibitem{Costello} K. Costello, 
  {\em Bilinear and quadratic variants on the Littlewood-Offord problem}, 
  submitted (2009).
  
\bibitem{CTV} K. Costello, T. Tao, V. Vu,
  {\em Random symmetric matrices are almost surely non-singular},  
  Duke Math. J. 135 (2006), 395--413.

\bibitem{ESY} L. Erd\"os, B. Schlein, H.-T. Yau, 
  {\em Wegner estimate and level repulsion for Wigner random matrices},
  Int. Math. Res. Not. 3 (2010), 436--479.
  
\bibitem{Gotze} F. G\"otze, 
  {\em Asymptotic expansions for bivariate von Mises functionals}, 
  Z. Wahrsch. Verw. Gebiete 50 (1979), 333--355.

\bibitem{KKS} J. Kahn, J. Koml\'os and E. Szemer\'edi,
  {\em On the probability that a random $\pm 1$-matrix is singular}, 
  J. Amer. Math. Soc. 8 (1995), 223--240.

\bibitem{Latala} R. Latala, 
  {\em Some estimates of norms of random matrices}, 
  Proc. Amer. Math. Soc. 133 (2005), 1273--1282.

\bibitem{LT}    M. Ledoux, M. Talagrand,
   {\em Probability in Banach spaces},
   Springer, 1991.
   
\bibitem{Nguyen} H. Nguyen,
  {\em Inverse Littlewood-Offord problems and the singularity of random symmetric matrices},
  Duke Math. J., to appear. 
  
\bibitem{Nguyen small} H. Nguyen, 
  {\em On the least singular value of random symmetric matrices}, 
  submitted.

\bibitem{R} M. Rudelson,
  {\em Invertibility of random matrices: norm of the inverse},
  Annals of Mathematics 168 (2008), 575--600.

\bibitem{RV square} M. Rudelson, R. Vershynin,  
  {\em The Littlewood-Offord Problem and invertibility of random matrices}, 
  Advances in Mathematics 218 (2008), 600--633.
  
\bibitem{RV rectangular} M. Rudelson, R. Vershynin,  
  {\em Smallest singular value of a random rectangular matrix}, 
  Communications on Pure and Applied Mathematics 62 (2009), 1707--1739.
  
\bibitem{RV ICM} M. Rudelson, R. Vershynin, 
  {\em Non-asymptotic theory of random matrices: extreme singular values}, 
  Proceedings of the International Congress of Mathematicians, Hyderabad, India, 2010.
  
\bibitem{Sidorenko} A. Sidorenko, 
  {\em A correlation inequality for bipartite graphs}, 
  Graphs Combin. 9 (1993), 201--204
  
\bibitem{TV} T. Tao, V. Vu,
  {\em Additive combinatorics.}
  Cambridge Studies in Advanced Mathematics, 105.
  Cambridge University Press, Cambridge, 2006.

\bibitem{TV Bulletin} T. Tao, V. Vu,
  {\em From the Littlewood-Offord problem to the Circular Law: universality of the spectral distribution of random matrices},
  Bull. Amer. Math. Soc. 46 (2009), 377--396.

\bibitem{TV Annals} T. Tao, V. Vu, 
  {\em Inverse Littlewood-Offord theorems and the condition number of random discrete matrices}, 
  Ann. of Math. (2) 169 (2009), 595--632.
    
\bibitem{TV smallest} T. Tao, V. Vu, 
  {\em Random matrices: the distribution of the smallest singular values},  
  Geom. Funct. Anal.  20  (2010), 260--297.
  
\bibitem{TV universality} T. Tao, V. Vu, 
  {\em Random matrices: universality of local eigenvalue statistics},
  Acta Math. 206 (2011), 127--204.

\bibitem{TV universality edge} T. Tao, V. Vu, 
  {\em Random matrices: Universality of local eigenvalue statistics up to the edge}, 
  Comm. Math. Phys. 298 (2010), 549--572.

\bibitem{TV localization} T. Tao, V. Vu, 
  {\em Random matrices: Localization of the eigenvalues and the necessity of four moments}, 
  Acta Mathematica Vietnamica 36 (2011), 431--449.
  
\bibitem{V intro rmt} R. Vershynin, 
  {\em Introduction to the non-asymptotic analysis of random matrices}, 
  in: Compressed sensing: theory and applications, eds. Y.~Eldar and G.~Kutyniok, 
  Cambridge University Press, to appear.

\end{thebibliography}
\end{document}